\documentclass[12pt]{amsart}
\usepackage[margin=1.0in]{geometry}
\usepackage{amsmath,amssymb,amsfonts,graphicx,color,fancyhdr,psfrag,comment,enumerate}
\usepackage[latin1]{inputenc}
\usepackage[hyperpageref]{backref}
\usepackage[colorlinks=true, pdfstartview=FitV, linkcolor=blue, citecolor=blue, urlcolor=blue]{hyperref}
\usepackage[capitalise,noabbrev]{cleveref}
\usepackage{tcolorbox,longfbox}
\allowdisplaybreaks
\usepackage{epstopdf}
\usepackage{mathrsfs}
\usepackage{epsfig}
\usepackage{hyperref}

\usepackage{ifthen}
\usepackage{float}
\usepackage{enumerate}
\usepackage{mathtools}
\usepackage[bottom]{footmisc}
\usepackage{tikz}
\usepackage{comment}
\usetikzlibrary{matrix,shapes,arrows,positioning,chains}
\usepackage{amssymb,amsmath,amsfonts}
\usepackage{bbm}

\numberwithin{equation}{section}

\newtheorem{theorem}{Theorem}[section]

\newtheorem{lemma}[theorem]{Lemma}

\renewcommand{\Im}{\operatorname{Im}}

\newcommand{\R}{\mathbb{R}}
\newcommand{\C}{\mathbb{C}}

\begin{document}
\title[The twisted Laplacian with drift]{Riesz transforms associated with the twisted Laplacian with drift}

\author[N. Garg and R. Garg]
{Nishta Garg \and Rahul Garg}

\address[N. Garg]{Department of Mathematics, Indian Institute of Science Education and Research Bhopal, Bhopal--462066, Madhya Pradesh, India.}
\email{nishta21@iiserb.ac.in}

\address[R. Garg]{Department of Mathematics, Indian Institute of Science Education and Research Bhopal, Bhopal--462066, Madhya Pradesh, India.}
\email{rahulgarg@iiserb.ac.in}


\begin{abstract}
We consider the Riesz transforms of arbitrary order associated with the twisted Laplacian with drift on $\C^n$ and study their strong-type $(p, p)$, $1<p<\infty$, and weak-type $(1, 1)$ boundedness. 
\end{abstract}
\maketitle

\section{Introduction}
The motivation for the present work stems from a series of recent works \cite{Li-Sjogren-Wu-drift-Euclidean-Math-Z-2016, Li-Sjogren-drift-real-hyperbolic-Potential-Anal-2017, Li-Sjogren-drift-sharp-endpoint-Euclidean-Canad-2021, Li-Sjogren-drift-Heisenberg-JFAA-2022}, where the authors studied the weak-type $(1, 1)$ boundedness of the Riesz transforms associated with the Laplacian (or the sub-Laplacian) with drift on $\R^n$, real hyperbolic spaces and the Heisenberg groups. 

\medskip On $\R^n$, equipped with the standard distance, let $\nabla = \left( \partial_1, \ldots, \partial_n \right)$ be  the first order gradient vector field and  $\Delta = \sum_{j=1}^n \partial_j^2$ be the standard Laplacian. Given a non-zero vector $\nu \in \R^n$, let the Laplacian with drift $\nu$ be defined by 
$$ \Delta_\nu = \Delta + 2 \, \nu \cdot \nabla.$$
Consider the space $L^2(\R^n, d\mu_\nu)$ formed using the exponentially growing measure $d\mu_\nu (x) = e^{2 \nu \cdot x} \, dx.$ Then it turns out that $- \Delta_\nu$ is positive-definite and essentially self-adjoint on $L^2(\R^n, d\mu_\nu)$, and the first order Riesz transform $\nabla (-\Delta_\nu)^{-1/2}$ is an isometry on $L^2(\R^n, d\mu_\nu)$. In 2004, Lohou\'{e}--Mustapha \cite{Lohoue-Mustapha-drift-Euclidean-Trans-AMS-2004} proved that the Riesz transforms $\frac{\partial^\alpha}{\partial x^\alpha} (-\Delta_\nu)^{-k/2}$ of arbitrary order $|\alpha| = k \geq 1$ are bounded on $L^p(\mathbb{R}^n, \, d \mu_\nu)$ for every $1 < p < \infty$ (see Theorems 1 and 2 in \cite{Lohoue-Mustapha-drift-Euclidean-Trans-AMS-2004}). In fact, the results proved by Lohou\'{e}--Mustapha hold for a wide class of Lie groups. 

\medskip The weak-type $(1,1)$ boundedness of the Riesz transforms  was not  considered in the literature until very recently, when in 2016 Li--Sj\"ogren--Wu \cite{Li-Sjogren-Wu-drift-Euclidean-Math-Z-2016} proved the following result. 
\begin{theorem}[Li--Sj\"ogren--Wu \cite{Li-Sjogren-Wu-drift-Euclidean-Math-Z-2016}] \label{thm:LSW-Riesz-first-order-Euclidean}
The first order Riesz transform $\nabla(-\Delta_\nu)^{-1/2}$ is of weak-type $(1,1)$ on $\mathbb{R}^n$ with respect to the measure $d \mu_\nu$, uniformly in $\nu.$ 
\end{theorem}

In 2021, Li--Sj\"ogren \cite{Li-Sjogren-drift-sharp-endpoint-Euclidean-Canad-2021} extended the analysis of \cite{Li-Sjogren-Wu-drift-Euclidean-Math-Z-2016} to higher order Riesz transforms and established the following beautiful result. Let $D$ be a homogeneous differential operator of degree $k \geq 1$ on $\R^n$ with constant coefficients. Consider the $k^{th}$-order Riesz transform $R_D = D(-\Delta_\nu)^{-k/2}$. With $\partial_\nu$ denoting the  derivative in the direction of the vector $\nu$, one can express $D = \sum_{j=0}^k \partial_\nu^j D'_{k-j}$, where $D'_{k-j}$ denote the constant coefficient homogeneous differential operators of order $k-j$ involving differentiation only in the directions orthogonal to $\nu$. Define $q$ to be the number denoting the highest order of differentiation in the operator $D$ in the direction of $\nu$, in the sense that $q = \max \{j : D'_{k-j} \neq 0\} \in \{0, 1, \ldots, k\}$. 
\begin{theorem}[Li--Sj\"ogren \cite{Li-Sjogren-drift-sharp-endpoint-Euclidean-Canad-2021}] \label{thm:LS-Riesz-higher-order-sharp-Euclidean}
The Riesz transform $R_D=D(-\Delta_\nu)^{-k/2}$ is of weak-type $(1,1)$ if and only if $q \leq 2.$ On the other hand, when $q\geq 3$, there exists a constant $C=C(\nu,D)$ such that 
\begin{equation} \label{ineq:thm-LSW-higher-order-sharp-Euclidean}
\mu_\nu \{x:|R_Df(x)|>\alpha\} \leq C \int_{\mathbb{R}^n} \frac{|f|}{\alpha} \left(1+ln^+\frac{|f|}{\alpha}\right)^{\frac{q}{2}-1} d\mu_\nu.
\end{equation} 
for all $f \in L(1+ln^+L)^{\frac{q}{2}-1} (\R^n, d\mu_\nu)$ and all $\alpha>0$. 

Moreover, inequality \eqref{ineq:thm-LSW-higher-order-sharp-Euclidean} is sharp in the sense that it fails for $L(1+ln^+L)^r (\R^n, d\mu_\nu)$  for any $r < \frac{q}{2}-1$. 
\end{theorem}

Moving on to the Heisenberg group $\mathbb{H}^n = \R^n \times \R^n \times \R$, we consider the sub-Laplacian $ \mathcal{L} = \sum_{j=1}^n (\mathcal{X}_j^2 + \mathcal{Y}_j^2) $ where $ \mathcal{X}_j, \, \mathcal{Y}_j $ are certain left invariant vector fields playing the role of $ \partial_j $ on $ \R^n.$ We define the sub-Laplacian with non-zero drift $\nu = (a,b) \in \R^n \times \R^n$ by $ \mathcal{L}_\nu = \mathcal{L}+2\sum_{j=1}^n ( a_j \mathcal{X}_j + b_j \mathcal{Y}_j).$ As it turns out, this is self adjoint on $ L^2(\mathbb H^n,d\mu_\nu) $ where $d\mu_\nu (x,y,t) = e^{2 (a \cdot x + b \cdot y)} \, dx \, dy \, dt$. Consider  $\widetilde{\mathcal{R}}_{k, \nu},$ the Riesz transforms of order $k \geq 1$ on $\mathbb{H}^n$ associated  to the sub-Laplacian with drift to be defined later (see Subsection \ref{subsec:setup-definitions-with-drift}, leading to definition \eqref{def:Riesz-transform-gen-form-Heisenberg}). As mentioned earlier, the following result on the $L^p$-boundedness of $\widetilde{\mathcal{R}}_{k, \nu}$ is due to Lohou\'{e}--Mustapha (see Theorem 2 in \cite{Lohoue-Mustapha-drift-Euclidean-Trans-AMS-2004}). 
\begin{theorem}[Lohou\'{e}--Mustapha \cite{Lohoue-Mustapha-drift-Euclidean-Trans-AMS-2004}] \label{thm:LM-Lp-Riesz-Heisenberg}
For any $1<p<\infty$ and $k \geq 1$, the Riesz transforms $\widetilde{\mathcal{R}}_{k, \nu}$ are bounded on $L^p(\mathbb{H}^n,d \mu_\nu)$, uniformly in $\nu$. 
\end{theorem}

Although in \cite{Lohoue-Mustapha-drift-Euclidean-Trans-AMS-2004}  the uniformity of the operator norm on the drift vector $\nu$ is not mentioned, by making use of the dilations and rotations on $\mathbb{H}^n$ it is very easy to verify that the operator norms are indeed independent of  the  drift vector $\nu$.

\medskip Recently, Li--Sj\"ogren \cite{Li-Sjogren-drift-Heisenberg-JFAA-2022} studied analogue of Theorem \ref{thm:LSW-Riesz-first-order-Euclidean} and a partial analogue of Theorem \ref{thm:LS-Riesz-higher-order-sharp-Euclidean} on the Heisenberg group $\mathbb{H}^n$. 
\begin{theorem}[Li--Sj\"ogren \cite{Li-Sjogren-drift-Heisenberg-JFAA-2022}] \label{thm:LS-Riesz-Heisenberg}
The first order Riesz transforms $\widetilde{\mathcal{R}}_{1, \nu}$ are of weak-type $(1,1)$ with respect to $d\mu_\nu$, uniformly in $\nu$. But, for $k \geq 3$, not all the Riesz transforms $\widetilde{\mathcal{R}}_{k, \nu}$ are of weak-type $(1,1)$ with respect to $d\mu_\nu.$
\end{theorem}

Note that Theorem \ref{thm:LS-Riesz-Heisenberg} does not talk about the second order Riesz transforms and to the best of our knowledge it is not yet known whether they are of weak-type $(1,1)$ or not. Moreover, it is also an interesting open question to have an analogue of Theorem \ref{thm:LS-Riesz-higher-order-sharp-Euclidean} dealing with the number of derivatives in the direction of the drift vector $\nu$. 

\medskip In the present work, we study analogues of these results for the twisted Laplacians (or the special Hermite operators) $L_\lambda$ on $\C^n$ for $\lambda \in \mathbb{R} \setminus \{0\}$. These operators $L_\lambda$ are on the one hand intimately connected to the sub-Laplacian $\mathcal{L}$ on $\mathbb{H}^n$ as they are defined by the relation $  L_\lambda f(z) = e^{-i\lambda t}\, \mathcal{L}\left(f(z)e^{i\lambda t}\right)$ and on the other hand reduce  to the standard  Laplacian on $\C^n (\cong \R^{2n})$ as $ \lambda \rightarrow 0.$  Unlike the Laplacian or the sub-Laplacian, the twisted Laplacian is not dilation invariant which makes the situation a little bit different from the above cases.

\medskip We now state our main results. Consider the first order gradient vector fields $X_j (\lambda)$ and $Y_j (\lambda)$ by setting $\displaystyle X_j(\lambda) F(z) = e^{-i\lambda s} \mathcal{X}_j \left( e^{i \lambda s} F(z) \right), \, Y_j(\lambda) F(z) = e^{-i\lambda s} \mathcal{Y}_j \left( e^{i \lambda s} F(z) \right),$ so that $L_\lambda = \sum_{j=1}^{n} \left( X_j(\lambda)^2 + Y_j(\lambda)^2 \right)$. For a non-zero vector $\nu = (a, b) \in \R^n \times \R^n$, consider the twisted Laplacian with drift $\nu$, given by 
$$L_{\nu,\lambda} = L_\lambda + 2 \, \nu \cdot \nabla(\lambda) = L_\lambda + 2 \sum_{j=1}^n \left(a_j X_j(\lambda) + b_j Y_j(\lambda) \right),$$
where $\nabla (\lambda) = (X_1 (\lambda), X_2 (\lambda), \ldots, X_n (\lambda), Y_1 (\lambda), Y_2 (\lambda), \ldots, Y_n (\lambda))$. 

\medskip It can be easily verified that the operator $L_{\nu,\lambda}$ is negative-definite and essentially self-adjoint on $L^2(\C^n, d\mu_\nu)$ where $d\mu_\nu(x, y) = e^{2(a \cdot x + b \cdot y)} \, dx \, dy$. 

\medskip Consider a monomial of the type $P_k(\lambda) = \theta_1(\lambda) \theta_2(\lambda) \cdots \theta_{k}(\lambda)$ with $\theta_j(\lambda)$ coming from the set $\{ X_1(\lambda), \ldots, X_n(\lambda), Y_1(\lambda), \ldots, Y_n(\lambda) \}$ for every $1 \leq j \leq k$. Denote by $\mathcal{R}_{k, \nu, \lambda}$ the $k^{th}$-order Riesz transform $P_k (\lambda) (-L_{\nu,\lambda})^{-k/2}$. Our first result is about the $L^p$-boundedness of these Riesz transforms. 
\begin{theorem} \label{thm:Lp-Riesz-twisted-Laplacian}
For any $1<p<\infty$ and $k \geq 1$, the Riesz transforms $\mathcal{R}_{k, \nu, \lambda}$ are bounded on $L^p(\C^n, d \mu_\nu)$, uniformly in $\lambda$ and $\nu$. 
\end{theorem}

Concerning the end-point estimates, the following theorem can be seen as an analogue of Theorem \ref{thm:LSW-Riesz-first-order-Euclidean} and a partial analogue of Theorem \ref{thm:LS-Riesz-higher-order-sharp-Euclidean}. For technical convenience, we take $\nu = e_1$. One can get the result for arbitrary $\nu \neq 0$ with the help of rotation and dilation which we discuss in Subsections \ref{subsec:drift-change-with-rotation} and  \ref{subsec:drift-change-with-scaling}. 
\begin{theorem} \label{thm:higher-order-Riesz-uniform-1-1-log-space}
Let $\mathcal{R}_{k, e_1, \lambda}$ be the $k^{th}$-order Riesz transforms. Let $k_j$ be the total power of $X_j(\lambda)$ and $Y_j(\lambda)$ in the monomial $P_k(\lambda)$ defining $\mathcal{R}_{k, e_1, \lambda}$, so that $\sum_{j=1}^n k_j = k.$ 
\begin{enumerate}
\item For $k_1 \in \{0, 1, 2\},$ $\mathcal{R}_{k, e_1,\lambda}$ is weak-type $(1,1)$, uniformly in $\lambda$ and $\nu$. 

\item For any $k_1 \geq 3$, there exists a constant $C_k > 0$ such that 
\begin{align} \label{ineq:thm-LlogL} 
\mu_{e_1} \{z:|\mathcal{R}_{k, e_1, \lambda} f(z)|>\alpha\} \leq C_k \int_{\mathbb{C}^n} \frac{|f|}{\alpha} \left(1+ln^+\frac{|f|}{\alpha} \right)^{\frac{k_1}{2}-1} d\mu_{e_1}, 
\end{align}
for every $\alpha>0$, $\lambda \neq 0$, and $f \in L(1+ln^+L)^{\frac{k_1}{2}-1} (\mathbb{C}^n, d\mu_{e_1})$. 

Moreover, inequality \eqref{ineq:thm-LlogL} is sharp in the following sense. There exists a monomial such that for the corresponding Riesz transform $\mathcal{R}_{k_1, e_1, \lambda}$, the inequality does not hold true uniformly in $\lambda$ for the spaces $L(1+ln^+L)^r (\mathbb{C}^n, d\mu_{e_1})$ for any $r < \frac{k_1}{2}-1$. 
\end{enumerate} 
\end{theorem}

For  Riesz transforms of order three or more, we have the following results. 
\begin{theorem} \label{thm:higher-order-not-uniform}
Let $P_k(\lambda) = (\nu \cdot \nabla (\lambda))^k$. If $k \geq 3$, then the weak-type $(1,1)$ boundedness of the Riesz transform $P_k (\lambda) (-L_{\nu,\lambda})^{-k/2}$ is not uniform in $\lambda$ and $\nu$. 
\end{theorem} 

We remark that $ \left\| \mathcal{R}_{k, \nu, \lambda} \right\|_{L^{1,\infty} (\mathbb{C}^n, \, d\mu_\nu)} = \left\| \mathcal{R}_{k, \nu', \lambda/|\nu|^2} \right\|_{L^{1,\infty} \left(\mathbb{C}^n, \, d\mu_{ \nu'} \right)}$ where $\nu' = |\nu|^{-1} \nu$ (see Lemma \ref{lem:Riesz-transform-drift-in-scale-transformation}). Keeping $\lambda/|\nu|^2$ away from $0$ and exploiting the spectral properties of $L_{\lambda/|\nu|^2}$, we do have the boundedness of Riesz transforms of arbitrarily high order. 
\begin{theorem} \label{thm:higher-order-bdd-drift}
For any $M > 0$ and $k \geq 3$, there exists a constant $C_{M, k} > 0$ such that for any $\lambda$ and $\nu$ satisfying  $\frac{|\lambda|}{|\nu|^2} \geq M$, we have 
$$ \alpha \, \mu_\nu \left\{ z : |\mathcal{R}_{k,\nu,\lambda} f(z)|>\alpha \right\} \leq C_{M, k} \int_{\mathbb{C}^n} |f| \, d\mu_\nu, $$ 
for all $\alpha > 0$. 
\end{theorem}

Before moving on, let us recall that the Laplacian (or the sub-Laplacian) with drift satisfies the small-time Gaussian-type heat kernel bounds in a number of cases (of Lie groups). But the measures on the associated spaces are known to have exponential growth and hence such spaces are not of homogeneous type in the sense of Coifman--Weiss \cite{Coifman-Weiss-book-1971}. On such spaces, the analysis of singular integral operators, such as the Riesz transforms and spectral multipliers is very delicate. In the past two decades there has been a lot of interest in this direction and thus the literature is ever growing. We refer the interested readers to \cite{Alexopoulos-drift-Lie-groups-Memoirs-AMS-2002, Lohoue-Mustapha-drift-Euclidean-Trans-AMS-2004, Hebisch-Mauceri-Meda-spectral-multipliers-drift-Lie-group-Math-Z-2005, Martini-Ottazzi-Vallarino-spectral-multipliers-drift-Lie-group-Revista-2019} and various references therein. 

\medskip \noindent \textbf{Organisation of the paper:} This article is organised as follows. We collect all the preliminary material in Section \ref{sec:preliminaries-and-reductions}. 
Theorems \ref{thm:Lp-Riesz-twisted-Laplacian} and \ref{thm:higher-order-not-uniform} are proved in Section \ref{sec:proof-thm-with-transference}. We begin Section \ref{sec:kernel-estimates} with estimating the kernels of the Riesz transforms which are then used in the proofs of Theorems \ref{thm:higher-order-Riesz-uniform-1-1-log-space} and \ref{thm:higher-order-bdd-drift}. 


\section{Preliminaries and basic reductions} \label{sec:preliminaries-and-reductions}


\subsection{Twisted Laplacian on \texorpdfstring{$\C^n$}{}} \label{subsec:basic-setup-definitions}
We recall the basic results on the sub-Laplacian $\mathcal{L}$ on $\mathbb{H}^n$ and the twisted Laplacian $L_\lambda$ on $\C^n$ which we require in this work. For more details we refer to the monograph \cite{Thangavelu-uncertainty-principle-book}. 

\medskip Consider the Heisenberg group $\mathbb{H}^n = \C^n \times \R$ together with the group law
$$ (z,s) (w,s') = \left( z+w, s + s' + \frac{1}{2} \Im (z \cdot \bar{w}) \right). $$
It is well known that $\mathbb{H}^n$ is a unimodular group with the Lebesgue measure $dz \, ds$ of $\C^n \times \R$ being its Haar measure. The convolution of two functions  $f_1, f_2 \in L^1(\mathbb{H}^{n})$ is defined by
$$ f_1 \ast f_2 (z,s) = \int_{\mathbb{H}^{n}} f_1 \left((z,s) (w,s')^{-1}\right) f_2(w,s') \, dw \, ds'.$$ 

\medskip For any $\lambda \in \R$ we denote by $f^\lambda$ the inverse Euclidean Fourier transform (up to a constant) of $f$ in the last variable at $\lambda/(2 \pi)$, that is, $ f^\lambda(z) = \int_{\R} f(z,s) e^{i\lambda s} \, ds.$ It is then easy to verify that $ (f_1 \ast f_2)^\lambda = f_1^\lambda \ast_\lambda f_2^\lambda$ where the $\lambda$-twisted convolution of two functions $F_1, F_2 \in L^1(\C^n)$ is given by 
\begin{align*} 
(F_1 \ast_\lambda F_2)(z) = \int_{\C^n} F_1 (z-w) \, F_2 (w) \, e^{i \frac{\lambda}{2} \Im (z \cdot \bar{w})} \, dw. 
\end{align*} 
Note that when $\lambda = 0$, the $\lambda$-twisted convolution reduces to the standard Euclidean convolution on $\C^n (\cong \R^{2n})$. 

\medskip The sub-Laplacian $\mathcal{L}$ on $\mathbb{H}^{n}$ is defined by $ \mathcal{L} = \sum_{j=1}^n (\mathcal{X}_j^2 + \mathcal{Y}_j^2),$ where $\mathcal{X}_j$ and $\mathcal{Y}_j $ are the left invariant vector fields on $ \mathbb{H}^n$ which are given explicitly by 
$$ \mathcal{X}_j = \frac{\partial}{\partial x_j} +\frac{1}{2} y_j \frac{\partial}{\partial s} \quad \text{and} \quad \mathcal{Y}_j = \frac{\partial}{\partial y_j} - \frac{1}{2} x_j \frac{\partial}{\partial s}.$$
Let $q_t$ denote the integral kernel of the heat semigroup $\left( e^{t \mathcal{L}} \right)_{t>0}$ corresponding to the sub-Laplacian of the Heisenberg group $\mathbb{H}^n$ so that $e^{t \mathcal{L}} f = f \ast q_t.$ It is well known that 
\begin{align} \label{def:heat-kernel-twisted-laplacian}
q_t^\lambda (z) = (4\pi)^{-n} \left( \frac{\lambda}{ \sinh (\lambda t)}\right)^n e^{-\frac{1}{4}\lambda \coth (\lambda t) \, |z|^2}. 
\end{align}

\medskip The sub-Laplacian $ \mathcal{L} $ gives rise to a family of operators $ L_\lambda$ on $\C^n$, indexed by non-zero reals $ \lambda,$ which are called the twisted Laplacians (or the special Hermite operators). Given a smooth function $F$ on $\C^n, L_\lambda F $ is defined by the relation
$$ L_\lambda F(z) = e^{-i\lambda s} \mathcal{L} \left( e^{i \lambda s} F(z) \right).$$ 
Note also that $ (\mathcal{L}f)^\lambda(z) = L_\lambda f^\lambda(z),$ from which it follows that $L_\lambda$ is left invariant with respect to the $\lambda$-twisted convolution, that is, $L_\lambda (F_1 \ast_\lambda F_2) = (L_\lambda F_1) \ast_\lambda F_2.$ 

\medskip One defines the associated first order gradient vector fields $X_j (\lambda)$ and $Y_j (\lambda)$ by  setting
$$ X_j(\lambda) F(z) = e^{-i\lambda s} \mathcal{X}_j \left( e^{i \lambda s} F(z) \right), \quad \quad Y_j(\lambda) F(z) = e^{-i\lambda s} \mathcal{Y}_j \left( e^{i \lambda s} F(z) \right).$$ 
These are explicitly given by the expressions 
$$ X_j (\lambda) = \frac{\partial}{\partial x_j} + \frac{i \lambda}{2} y_j \quad , \quad Y_j (\lambda) = \frac{\partial}{\partial y_j} - \frac{i \lambda}{2} x_j,$$ 
and the operator $L_\lambda$ takes the form 
\begin{align} \label{formula:twisted-laplacian}
L_\lambda = \sum_{j=1}^{n} \left( X_j(\lambda)^2 + Y_j(\lambda)^2 \right) = \Delta_{z} - \frac{\lambda^2}{4}|z|^2 - i \lambda N.
\end{align} 
Here $\Delta_{z}$ is the standard Laplacian $\sum_{j=1}^{n} \left( \frac{\partial^2}{\partial x_j^2} + \frac{\partial^2}{\partial y_j^2} \right)$ on $\C^n$ and $ N = \sum_{j=1}^{n} \left( x_j \frac{\partial}{\partial y_j} - y_j \frac{\partial}{\partial x_j} \right)$. Note again that as $\lambda \to 0$, the twisted Laplacian $L_\lambda$ reduces to $\Delta_z$. 

\medskip For the diffusion semigroup $\left( e^{t \, L_\lambda} \right)_{t>0}$, the associated heat kernel $p_{t,\lambda}$ satisfies the relation $p_{t,\lambda} = q_t^\lambda$, with $q_t^\lambda$ given by \eqref{def:heat-kernel-twisted-laplacian}. It follows that $e^{t \, L_\lambda} F = F \ast_\lambda p_{t,\lambda}.$ 


\subsection{Twisted Laplacian on \texorpdfstring{$\C^n$}{} with drift} \label{subsec:setup-definitions-with-drift}
On a suitable weighted manifold $(M, g, \mu)$, suppose we have a smooth positive function $h$  such that $\Delta_\mu h + \alpha h = 0$ for some real constant $\alpha$. Then for the measure $d\tilde{\mu} = h^2 d\mu$ on $M$, one can obtain the heat semigroup and the heat kernel for the Laplacian with drift $\Delta_{\tilde{\mu}} = \frac{1}{h} \circ (\Delta_\mu + \alpha) \circ h$ on $(M, g, \tilde{\mu})$, with the help of that for the Laplacian $\Delta_\mu$ on $(M, g, \mu)$ (see \S 9.2.4 in \cite{Alex-Heat-ker-analysis-on-manifols}). The operator $\Delta_{\tilde{\mu}}$ is self-adjoint with respect to the measure $d\tilde{\mu}$. 

\medskip For the sub-Laplacian $\mathcal{L}$ on the Heisenberg group $\mathbb{H}^n$, if one takes $h(x,y,t) = e^{a \cdot x + b \cdot y} $ for a fixed non-zero vector $\nu = (a, b) \in \R^n \times \R^n$, then it is easy to see that the associated sub-Laplacian with drift $\mathcal{L}_\nu$ is given by  
\begin{align}
\mathcal{L}_\nu = \mathcal{L} + 2 \sum_{j = 1}^n \left( a_j \mathcal{X}_j + b_j \mathcal{Y}_j \right). 
\end{align}
With $q_t$ denoting the integral kernel of the heat semigroup $\left( e^{t \mathcal{L}} \right)_{t>0}$ as in \eqref{def:heat-kernel-twisted-laplacian}, the integral kernel $q_t^{(\nu)}$ of the heat semigroup $\left( e^{t \mathcal{L}_\nu} \right)_{t>0}$ takes the form 
$$ q_t^{(\nu)}(g,g^\prime)= e^{-t |\nu|^2} \, e^{-a \cdot (x+u) - b \cdot (y+v)} \, q_t((g^\prime)^{-1}g) $$
for $g = (z, s) = (x, y, s), \, g' = (w, s') = (u, v, s') \in \mathbb{H}^n$. 

\medskip Let $Q_k$ stand for a homogeneous polynomial of degree $k \geq 1$ in $\mathcal{X}_j$ and $\mathcal{Y}_j, \, 1 \leq j \leq n$. More precisely, $Q_k$ is a finite linear combination of monomials of the type $\theta_1 \theta_2 \cdots \theta_{k}$ with $\theta_j \in \{\mathcal{X}_1, \ldots, \mathcal{X}_n, \mathcal{Y}_1, \ldots , \mathcal{Y}_n\}$ for every $1 \leq j \leq k$. We associate with it the $k^{th}$-order Riesz transform $\widetilde{\mathcal{R}}_{Q_k, \nu}$ defined by 
\begin{align} \label{def:Riesz-transform-gen-form-Heisenberg} 
\widetilde{\mathcal{R}}_{Q_k, \nu} = Q_k (-\mathcal{L}_{\nu})^{-k/2}.
\end{align}
In particular, with $Q_k$ being any monomial of the above type, we use the single notation $\widetilde{\mathcal{R}}_{k, \nu} $ for the Riesz transform. 

\medskip We utilise the relation between the sub-Laplacian $\mathcal{L}$ and the twisted Laplacian $L_\lambda$ for defining the twisted Laplacian with  drift. To be more precise, we define the twisted Laplacian $L_{\nu,\lambda}$ with drift $\nu$ by 
\begin{align} \label{eq:relation-drift-operators}
L_{\nu,\lambda} F(z) = e^{-i \lambda s} \mathcal{L}_\nu (e^{i \lambda s} F(z)). 
\end{align}
With $\nabla (\lambda) = (X_1 (\lambda), X_2 (\lambda), \ldots, X_n (\lambda), Y_1 (\lambda), Y_2 (\lambda), \ldots, Y_n (\lambda))$, it is straightforward to verify that 
\begin{align} \label{special-Hermite-operator-with-drift}
L_{\nu,\lambda} = L_\lambda + 2 \, \nu \cdot \nabla(\lambda) = L_\lambda + 2 \sum_{j=1}^n \left(a_j X_j(\lambda) + b_j Y_j(\lambda) \right).
\end{align}
By defining $d\mu_\nu(x, y) = e^{2(a \cdot x + b \cdot y)} \, dx \, dy$, it is also easy to check that the pair $(L_{\nu,\lambda}, d\mu_\nu)$ satisfies the Green's formula: 
\begin{align} \label{identity-green-drifted-twisted-laplacian} 
& \int F_1(z) \, \overline{\left(L_{\nu,\lambda} F_2 \right)(z)} \, d\mu_\nu(z) = \int \left(L_{\nu,\lambda} F_1 \right)(z) \, \overline{F_2(z)} \, d\mu_\nu(z) \\ 
\nonumber &= - \sum_{j=1}^n \int \left\{ \left(X_j(\lambda) F_1\right) (z) \, \overline{\left( X_j(\lambda) F_2 \right)(z)} + \left( Y_j(\lambda) F_1 \right) (z) \, \overline{\left( Y_j(\lambda) F_2 \right)(z)} \right\} d\mu_\nu(z), 
\end{align}
for any $F_1, F_2 \in C_c^\infty(\C^n).$ This simply means that $L_{\nu,\lambda}$ is negative-definite and essentially self-adjoint on $L^2(\C^n, d\mu_\nu)$. 

\medskip Using the functional calculus, we get from \eqref{eq:relation-drift-operators} that 
\begin{align*} 
e^{t \mathcal{L}_\nu} (e^{i \lambda s} F(z))= e^{i \lambda s} e^{t \, L_{\nu,\lambda}} F(z), 
\end{align*}
and thus if $p_{t,\lambda}^{(\nu)}$ denotes the integral kernel of $e^{t \, L_{\nu,\lambda}}$, in the sense that 
$$ e^{t \, L_{\nu,\lambda}} \phi(z) = \int_{\C^n} \phi(w) \, p_{t,\lambda}^{(\nu)} (z,w) \, d\mu_{\nu}(w),$$ 
then it follows that 
\begin{align} \label{eq:Heat-ker-twisted-lap-drift} 
p^{(\nu)}_{t,\lambda} (z,w) = e^{-|\nu|^2t} \, e^{-a \cdot (x+u) - b \cdot (y+v)} \, p_{t,\lambda}(z-w) \, e^{\frac{i\lambda}{2} \Im(z \cdot \bar{w})}, 
\end{align}
where $\nu = (a, b)$, $z = x+iy$ and $w = u+iv$. Also, the integral kernel of $(-L_{\nu,\lambda})^{-k/2}$, for $k \geq 1,$ can be expressed as 
\begin{equation} \label{eq:kernel-fractional-power-twisted-lap-drift} 
(-L_{\nu,\lambda})^{-k/2} (z,w) = \frac{1}{\Gamma(k/2)} \int_{0}^{\infty} t^{\frac{k}{2}-1} \, p_{t,\lambda}^{(\nu)}(z,w) \, dt. 
\end{equation}

\medskip As in the case of $\mathbb{H}^n$, working with the twisted Laplacians $L_{\nu,\lambda}$ on $\mathbb{C}^n$, let $P_k (\lambda)$ denote a homogeneous polynomial of degree $k \geq 1$ in $X_j (\lambda)$ and $Y_j (\lambda)$, with $1 \leq j \leq n$. We consider the $k^{th}$-order Riesz transform $\mathcal{R}_{P_k, \nu,\lambda}$ defined by 
\begin{align} \label{def:twisted-Riesz-transform-gen-form} 
\mathcal{R}_{P_k, \nu, \lambda} = P_k (\lambda) (-L_{\nu,\lambda})^{-k/2}, 
\end{align}
and as above, for monomials in $X_j (\lambda)$ and $Y_j (\lambda)$, we use the single notation $\mathcal{R}_{k, \nu, \lambda}$. 

\medskip It is straightforward to verify that the Riesz transforms $\widetilde{\mathcal{R}}_{Q_k, \nu}$ and the family of Riesz transforms $\mathcal{R}_{P_k, \nu,\lambda}$ are in a one-one correspondence. More precisely, one can go back and forth via the recipe 
\begin{align} \label{relation:Riesz-transform-gen-form-Cn-Hn} 
\widetilde{\mathcal{R}}_{Q_k, \nu} f (z,t) = \int_{\R} \left( \mathcal{R}_{P_k, \nu,\lambda} f^\lambda \right) (z) \, e^{-i \lambda t} \, d\lambda, 
\end{align}
and similarly one can establish a natural relation between $\widetilde{\mathcal{R}}_{k, \nu}$ and $\mathcal{R}_{k, \nu,\lambda}$. 

\medskip In order to simplify the technical side of some of the proofs of our results in the subsequent sections, we can perform rotation and dilation to reduce matters to the specific drift vectors. We show it over the next two subsections. 


\subsection{Drift and rotation} \label{subsec:drift-change-with-rotation}
Let us see here that to study the weak or strong-type boundedness of the Riesz transforms $R_{k,\nu,\lambda}$, without any loss of generality one can assume the drift vector $\nu$ to be of a specific form, say $\nu' = |\nu| e_1$. 

\medskip Given an $n \times n$ unitary matrix $A=(c_{ij})$, it is easy to see from the expression \eqref{eq:Heat-ker-twisted-lap-drift} that $ p_{t,\lambda}^{(A\nu)}(Az,Aw) = p_{t,\lambda}^{(\nu)}(z,w),$ which implies that $ (-L_{A\nu,\lambda})^{-1/2} (Az,Aw) = (-L_{\nu,\lambda})^{-1/2} (z, w).$ Now, via a simple chain rule of the differentiation we get  
\begin{align*}
& \left\{ X_j(\lambda) (-L_{\nu,\lambda})^{-1/2} \right\} (z,w) \\ 
&= \left\{ Z_j(\lambda) (-L_{\nu,\lambda})^{-1/2} \right\} (z,w) + \left\{ \bar{Z}_j(\lambda) (-L_{\nu,\lambda})^{-1/2} \right\} (z,w) \\ 
&= \left(\frac{\partial}{\partial z_j} - \frac{\lambda}{4} \bar{z}_j \right) \left\{(-L_{A\nu,\lambda})^{-1/2} (Az,Aw) \right\} + \left(\frac{\partial}{\partial \bar{z}_j} + \frac{\lambda}{4} z_j \right) \left\{(-L_{A\nu,\lambda})^{-1/2} (Az,Aw) \right\} \\ 
&= \sum_{l=1}^n c_{lj} \left( \frac{\partial}{\partial z_l} (-L_{A\nu,\lambda})^{-1/2} \right) (Az,Aw) - \frac{\lambda}{4} \bar{z}_j (-L_{A\nu,\lambda})^{-1/2} (Az,Aw) \\ 
& \quad + \sum_{l=1}^n \overline{c_{lj}} \left( \frac{\partial}{\partial \bar{z}_l} (-L_{A\nu,\lambda})^{-1/2} \right) (Az,Aw) + \frac{\lambda}{4} z_j (-L_{A\nu,\lambda})^{-1/2} (Az,Aw) \\ 
& = \left\{ \sum_{l=1}^n \left( c_{lj} \, Z_l(\lambda) + \overline{c_{lj}} \, \bar{Z}_l(\lambda) \right) (-L_{A \nu,\lambda})^{-1/2} \right\} (Az,Aw). 
\end{align*}

\medskip Recursively, one can verify that the integral kernel $\mathcal{R}_{k, \nu, \lambda} (z,w)$ can be expressed as a finite linear combination of the integral kernels of the form $\mathcal{R}_{k, A \nu, \lambda} (Az, Aw)$. The coefficients in such a linear combination involve the matrix entries $c_{ij}$'s which are uniformly bounded over the unitary matrices $A$. The claimed reduction then follows by choosing the unitary matrix $A$ such that $A \nu = \nu' = |\nu| e_1$. 


\subsection{Drift and scaling} \label{subsec:drift-change-with-scaling}
The following lemma is about the relation between scaling of the drift vector and the scaling factor in the twisted Laplacian. 
\begin{lemma} \label{lem:Riesz-transform-drift-in-scale-transformation} 
Given a Riesz transform $\mathcal{R}_{P_k, \nu, \lambda}$, with drift vector $\nu = (a,b)$, for any $s > 0$, 
$$ \left\| \mathcal{R}_{P_k, \nu, \lambda} \right\|_{L^{1,\infty} (\mathbb{C}^n, \, d\mu_\nu)} = \left\| \mathcal{R}_{P_k, s \nu, s^2 \lambda} \right\|_{L^{1,\infty} \left(\mathbb{C}^n, \, d\mu_{s \nu} \right)}. $$
\end{lemma} 
\begin{proof} 
Given $\phi \in C_c^\infty (\mathbb{R}^n)$, define $\psi(z) = \phi (sz)$, so that $ \left\| \psi \right\|_{L^1 \left(\mathbb{C}^n, \, d\mu_{s\nu} \right)} = s^{-2n} \left\| \phi \right\|_{L^1(\mathbb{C}^n, \, d\mu_\nu)},$ 
and therefore the claim of the lemma shall follow if we can show that 
\begin{equation} \label{eq:relation-riesz-transforms-dilation}
\mathcal{R}_{P_k, \nu, \lambda} \phi (sz) = \mathcal{R}_{P_k, s \nu, s^2 \lambda} \psi(z), 
\end{equation}
because then 
$$ \mu_\nu \left\{ z : \left| \mathcal{R}_{P_k, \nu,\lambda} \phi(z) \right| > \alpha \right\} = s^{2n} \, \mu_{s \nu} \left\{ \zeta : \left| \mathcal{R}_{P_k, s \nu, s^2 \lambda} \psi(\zeta) \right| > \alpha \right\}, $$
and therefore 
\begin{align*} \left\|\mathcal{R}_{P_k, \nu, \lambda}\right\|_{L^{1,\infty}(\mathbb{C}^n, \, d\mu_\nu)} & = \sup_{\alpha>0} \frac{\mu_\nu \left\{ z : |\mathcal{R}_{P_k, \nu, \lambda} \phi(z)| > \alpha \right\}}{\left\| \phi \right\|_{L^{1} (\mathbb{C}^n, \, d\mu_\nu)}} \\ 
& = \sup_{\alpha>0} \frac{\mu_{s\nu} \left\{ \zeta : |\mathcal{R}_{P_k, s \nu, s^2 \lambda} \psi(\zeta)|>\alpha \right\}}{\left\| \psi \right\|_{L^{1} \left(\mathbb{C}^n, \, d\mu_{s \nu} \right)}} = \left\| \mathcal{R}_{P_k, s \nu, s^2 \lambda} \right\|_{L^{1,\infty} \left(\mathbb{C}^n, \, d\mu_{s \nu} \right)}. 
\end{align*}

\medskip For technical convenience, we write below the proof of identity \eqref{eq:relation-riesz-transforms-dilation} only in the case of $k=1$ and with $ P_k (\lambda) = X_j (\lambda) $. From the proof it will become clear that the same idea works for general order $k$. 
Recall from \eqref{eq:kernel-fractional-power-twisted-lap-drift} that 
\begin{align*}
(-L_{\nu,\lambda})^{-1/2} (z, w) &= \frac{1}{\sqrt{\pi}} \int_{0}^{\infty} t^{-1/2} \, p_{t,\lambda}^{(\nu)}(z,w)  \, dt, 
\end{align*}
so that the integral kernel of $\mathcal{R}_{1, \nu,\lambda} = X_j (\lambda) (-L_{\nu,\lambda})^{-1/2} = \left( \frac{\partial}{\partial x_j} + \frac{i \lambda}{2} y_j \right) (L_{\nu,\lambda})^{-1/2} $ equals to 
$$ \mathcal{R}_{1, \nu,\lambda} (z,w) = \frac{1}{\sqrt{\pi}} \int_{0}^{\infty} t^{-1/2} \, \left(-a_j - \frac{\lambda}{2} \coth (\lambda t) (x_j-u_j) + \frac{i \lambda}{2} (y_j-v_j) \right) \, p_{t,\lambda}^{(\nu)}(z,w)  \, dt,  $$
and therefore  
\begin{align*}
\mathcal{R}_{1, \nu,\lambda} \phi(z) & = \frac{1}{\sqrt{\pi}} \int_{0}^{\infty} t^{-1/2} \int_{\C^n} \left(-a_j - \frac{\lambda}{2} \coth (\lambda t) (x_j-u_j) + \frac{i \lambda}{2} (y_j-v_j) \right) \, p_{t,\lambda}^{(\nu)}(z,w)  \, \phi(w) \, d\mu_\nu(w) \, dt. 
\end{align*}

\medskip In the above integrals, let us apply basic change of variables $t \mapsto s^2 t$, \, $w \mapsto s w$ and evaluate at the scaled point $sz$. Making use of the fact that $p_{s^2t,\lambda}^{(\nu)}(sz,sw) = s^{-2n} p_{t,s^2 \lambda}^{(s\nu)}(z,w)$, we get 
\begin{align*}
\mathcal{R}_{1, \nu, \lambda} \phi(sz) & = \frac{s^{2n+1}}{\sqrt{\pi}} \int_{0}^{\infty} t^{-1/2} \int_{\C^n} \left(-a_j - \frac{s \lambda}{2} \coth (s^2 \lambda t) \, (x_j-u_j) + \frac{i s \lambda}{2} (y_j-v_j) \right) \\ 
& \qquad \qquad p_{s^2t,\lambda}^{(\nu)}(sz, s w) \, \phi(s w) \, d\mu_{s\nu} (w) \, dt \\ 
& = \frac{1}{\sqrt{\pi}} \int_{0}^{\infty} t^{-1/2} \int_{\C^n} \left(- s a_j - \frac{s^2 \lambda}{2} \coth (s^2 \lambda t)  \, (x_j-u_j) + \frac{i s^2 \lambda}{2} (y_j-v_j) \right) \\ 
& \qquad \qquad p_{t,s^2 \lambda}^{(s\nu)}(z,w) \, \psi(w) \, d\mu_{s\nu} (w) \, dt \\ 
& = \mathcal{R}_{1, s \nu, s^2 \lambda} \psi(z), 
\end{align*}
which is exactly the claimed identity \eqref{eq:relation-riesz-transforms-dilation} and this completes the proof of the lemma. 
\end{proof} 

We end this section by noting that a minor modification in the proof of the above lemma leads to 
$$ \left\| \mathcal{R}_{P_k, \nu, \lambda} \right\|_{L^{p} (\mathbb{C}^n, \, d\mu_\nu)} = \left\| \mathcal{R}_{P_k, s \nu, s^2 \lambda} \right\|_{L^{p} \left(\mathbb{C}^n, \, d\mu_{s \nu} \right)},$$
for any $s > 0$ and $1 \leq p \leq \infty$. 


\section{Proofs of Theorems \ref{thm:Lp-Riesz-twisted-Laplacian} and \ref{thm:higher-order-not-uniform}} \label{sec:proof-thm-with-transference}

Given a homogeneous polynomial $P_k (\lambda)$ in $X_j (\lambda)$ and $Y_j (\lambda)$ of degree $k \geq 1$, we first study the behaviour of the Riesz transforms $\mathcal{R}_{P_k, \nu, \lambda}$ (as a function of $\lambda$) when applied to smooth functions. We shall make use of it in the proofs of Theorems \ref{thm:Lp-Riesz-twisted-Laplacian} and \ref{thm:higher-order-not-uniform} in Subsections \ref{subsec:proof-two-basic-thms-1} and \ref{subsec:proof-two-basic-thms-2}. 


\subsection{Continuity of the Riesz transforms \texorpdfstring{$\mathcal{R}_{P_k, \nu, \lambda}$}{} in \texorpdfstring{$\lambda$}{}-variable} \label{subsec:continuity-lambda-variable}
Fix $\phi \in C_c^\infty(\C^n)$ and define the functions $F : \mathbb{R} \setminus \{0\}  \times \C^n \to \mathbb{C}$ and $F_0 : \C^n \to \mathbb{C}$ by 
\begin{align} \label{eq:Riesz-lambda-pointwise-continuity}
F (\lambda, z) := \mathcal{R}_{P_k, \nu, \lambda} \phi (z) \quad \textup{and} \quad F_0 (z) := D_k \left( - \Delta_\nu \right)^{-k/2} \phi (z), 
\end{align}
where $D_k$ is the homogeneous differential operator obtained by taking $\lambda \to 0$ in $P_k(\lambda)$. 


\begin{lemma} \label{lem:Riesz-lambda-pointwise-continuity} 
For functions $F$ and $F_0$ defined by \eqref{eq:Riesz-lambda-pointwise-continuity}, the following assertions hold true. 

\begin{enumerate}[(i)]
\item $F_0$ is uniformly continuous on $\C^n$.  

\item $F$ is continuous on $\mathbb{R} \setminus \{0\} \times \C^n$. 

\item For any $z_0 \in \C^n$, $\displaystyle \lim_{(\lambda, z) \to (0, z_0)} F (\lambda, z) = F_0 (z_0).$ 
\end{enumerate} 
Consequently, $F$ extends to a continuous function on $\mathbb{R} \times \C^n$ by prescribing $F(0, z) := F_0 (z)$. 
\end{lemma}
\begin{proof} \textbf{Part $(i)$:} Let $D_k$ be given by $D_k = \sum_{|\alpha| + |\beta| = k} a_{\alpha, \beta} \, \partial_x^\alpha \, \partial_y^\beta$ and denote its associated homogeneous polynomial by $ p_k(\zeta) = p_k(\xi, \eta) = \sum_{|\alpha| + |\beta| = k} a_{\alpha, \beta} \xi^\alpha \, \eta^\beta$. For $t>0$, we have 
\begin{align*}
e^{t\Delta_\nu} \phi(z) & = (4\pi t)^{-n} e^{-|\nu|^2 t } \int_{\C^n}  e^{- a \cdot (x+u) - b \cdot (y+v)} \, e^{-\frac{1}{4t} |z-w|^2} \, \phi(w) \, d\mu_\nu(w) \\ 
& =  (4\pi t)^{-n} e^{-|\nu|^2 t } \int_{\C^n}  e^{- a \cdot (x-u) - b \cdot (y-v)} \, e^{-\frac{1}{4t} |z-w|^2} \, \phi(w) \, dw \\ 
& = \int_{\C^n} e^{-2\pi i(x+2at, \, y+2bt)\cdot (\xi,\eta)} \, e^{-4\pi^2 t|\zeta|^2} \, \widehat{\phi}(\zeta) \, d\zeta, 
\end{align*}
where the last equality follows from the Parseval's theorem for the Euclidean Fourier transform, and therefore  
\begin{align*}
F_0(z) & = D_k (-\Delta_\nu)^{-k/2} \phi(z) \\ 
& = \frac{1}{\Gamma(k/2)} \int_{0}^{\infty} t^{\frac{k}{2}-1} D_k e^{t\Delta_\nu} \phi(z) \, dt \\ 
& = \frac{(-2\pi i)^k}{\Gamma(k/2)} \int_{0}^{\infty} \int_{\C^n} t^{\frac{k}{2}-1} \, e^{-2\pi i(x+2at, \, y+2bt) \cdot (\xi,\eta)} \, e^{-4\pi^2 t|\zeta|^2} \, p_k(\zeta) \, \widehat{\phi}(\zeta) \, d\zeta \, dt. 
\end{align*}
Thanks to the homogeneity of the polynomial $p_k$, the above integral converges absolutely and the uniform continuity of $F_0$ follows from the dominated convergence theorem.  

\medskip \noindent \textbf{Part $(ii)$:} We shall proceed on the similar lines as in part $(i)$. Owing to the technical nature of the proof, we shall write the details for the case of $k=1$ with $P_k (\lambda) = X_j (\lambda)$, and it will be self-evident that the general case can be shown in an exactly same manner. Once again, for $t>0$ we have 
\begin{align*}
& e^{t \, L_{\nu,\lambda}} \phi(z) \\ 
\quad & = (4\pi)^{-n} \left(\frac{\lambda}{\sinh (\lambda t) } \right)^n e^{-|\nu|^2 t} \int_{\C^n} e^{- a \cdot (x-u) - b \cdot (y-v)} \, e^{-\frac{1}{4} \lambda \coth (\lambda t) \, |z-w|^2} \, \phi(w) \, e^{\frac{i\lambda}{2} \Im(z \cdot \bar{w})} \, dw \\ 
&= \left( \cosh{(\lambda t)} \right)^{-n} e^{-|\nu|^2 t} \, e^{\frac{|\nu|^2}{\lambda \coth (\lambda t)}} \\ 
& \qquad \times \int_{\C^n} e^{-2\pi i \left( x+\frac{2a}{\lambda \coth{(\lambda t)}}, \, y + \frac{2b}{\lambda\coth{(\lambda t)}} \right) \cdot (\xi,\eta)} \, e^{-\frac{4\pi^2|\zeta|^2}{\lambda\coth{(\lambda t)}}} \, \widehat{\phi} \left(\zeta+\frac{i\lambda z}{4\pi}\right) d\zeta. 
\end{align*}
with the last equality following from the Parseval's theorem for the Euclidean Fourier transform, and therefore
\begin{align} \label{eq:lim_lambda_and_z}
F(\lambda, z) & = X_j(\lambda) (-L_{\nu,\lambda})^{-1/2} \phi(z) \\ 
\nonumber & = \frac{1}{\Gamma(1/2)} \int_{0}^{\infty} t^{-1/2} \left( X_j(\lambda) \, e^{t \, L_{\nu,\lambda}} \, \phi \right) (z) \, dt \\ 
\nonumber & = \frac{1}{\Gamma(1/2)} \int_{0}^{\infty} \int_{\C^n} t^{-1/2} (\cosh {(\lambda t)})^{-n} e^{-|\nu|^2t} \, e^{\frac{|\nu|^2}{\lambda\coth{(\lambda t)}}} \, e^{-\frac{4\pi^2|\zeta|^2}{\lambda\coth{(\lambda t)}}} \\ 
\nonumber \nonumber & \qquad X_j (\lambda) \left\{ e^{-2\pi i\left(x+\frac{2a}{\lambda \coth {(\lambda t)}}, \, y+\frac{2b}{\lambda \coth {(\lambda t)}}\right) \cdot (\xi,\eta)} \, \widehat{\phi} \left( \xi - \frac{\lambda}{4\pi} y, \, \eta + \frac{\lambda}{4\pi} x \right) \right\} d\xi \, d\eta \, dt \\ 
\nonumber &= \frac{1}{\Gamma(1/2)} \int_{0}^{\infty} \int_{\C^n} t^{-1/2} (\cosh {(\lambda t)})^{-n} e^{-|\nu|^2t} \, e^{\frac{|\nu|^2}{\lambda \coth {(\lambda t)}}} \, e^{-2\pi i\left(x+\frac{2a}{\lambda \coth {(\lambda t)}}, \, y+\frac{2b}{\lambda\coth{(\lambda t)}}\right) \cdot (\xi,\eta)} \\ 
\nonumber & \qquad e^{-\frac{4\pi^2|\zeta|^2}{\lambda \coth {(\lambda t)}}} \left\{-2\pi i \, \xi_j \, \widehat{\phi} \left( \xi - \frac{\lambda}{4\pi} y, \, \eta + \frac{\lambda}{4\pi} x \right) + \frac{\lambda}{4\pi} \frac{\partial \widehat{\phi}}{\partial \eta_j} \left( \xi - \frac{\lambda}{4\pi} y, \, \eta + \frac{\lambda}{4\pi} x \right) \right. \\ 
\nonumber & \qquad \qquad \qquad \qquad \left. + \frac{i \lambda}{2} y_j \, \widehat{\phi} \left( \xi - \frac{\lambda}{4\pi} y, \, \eta + \frac{\lambda}{4\pi} x \right) \right\} d\xi \, d\eta\, dt.
\end{align}

\medskip Note that corresponding to each one of the three terms in the last expression, the integral is absolutely integrable. Moreover, one can apply the generalized Lebesgue dominated convergence theorem to pass on the limit $(\lambda, z)$ to $(\lambda_0, z_0)$ inside the integral $(0, \infty) \times \C^n$. This can be seen by analysing the following three integrals with their integrands dominating the corresponding ones in the above integral. Here we also make use of the basic estimate $\cosh {(\lambda t)} \gtrsim e^{t|\lambda|}$. 
\begin{align*}
I_1 (\lambda, z) &= \int_{0}^{\infty} \int_{\C^n} t^{-1/2} \, e^{-4\pi^2 t|\zeta|^2} \, |\zeta| \, \left| \widehat{\phi} \left(\zeta+\frac{i\lambda z}{4\pi}\right) \right| d\zeta \, dt, \\ 
I_2 (\lambda, z) &= |\lambda| \int_{0}^{\infty} \int_{\C^n} t^{-1/2} \, e^{-n t |\lambda|} \left| \frac{\partial \widehat{\phi}}{\partial \eta_j} \left(\zeta+\frac{i\lambda z}{4\pi}\right) \right| d\zeta \, dt, \\
\textup{and} \quad I_3 (\lambda, z) &= |\lambda| |z| \int_{0}^{\infty} \int_{\C^n} t^{-1/2} \, e^{-n t |\lambda|} \left| \widehat{\phi} \left(\zeta+\frac{i\lambda z}{4\pi}\right) \right| d\zeta \, dt. 
\end{align*}

\medskip Now, observe that 
\begin{align*}
I_1 (\lambda, z) &= \int_{0}^{\infty} \int_{\C^n} t^{-1/2} \, e^{-4\pi^2 t|\zeta|^2} |\zeta| \left| \widehat{\phi} \left(\zeta+\frac{i\lambda z}{4\pi}\right) \right| d\zeta \, dt \\ 
& = \int_{0}^{\infty} \int_{\C^n} t^{-1/2} \, e^{-4\pi^2 t \left|\zeta-\frac{i\lambda z}{4\pi}\right|^2} \left|\zeta-\frac{i\lambda z}{4\pi}\right| | \widehat{\phi}(\zeta) | \, d\zeta \, dt = \int_{0}^{\infty} \int_{\C^n} s^{-1/2} \, e^{-s} \, |\widehat{\phi}(\zeta)| \, d\zeta \, ds,
\end{align*}
which is finite independent of $\lambda$ and $z$-variables. 

\medskip Next, 
\begin{align*}
I_2 (\lambda, z) & = |\lambda| \int_{0}^{\infty} \int_{\C^n} t^{-1/2} \, e^{-nt|\lambda|} \left| \frac{\partial \widehat{\phi}}{\partial \eta_j} \left(\zeta+\frac{i\lambda z}{4\pi}\right) \right| d\zeta \, dt \\ 
& = |\lambda| \int_{0}^{\infty} \int_{\C^n} t^{-1/2} \, e^{-nt|\lambda|} \left| \frac{\partial \widehat{\phi}}{\partial \eta_j} (\zeta) \right| d\zeta \, dt = |\lambda|^{1/2} \int_{0}^{\infty} \int_{\C^n} s^{-1/2} e^{-s} \left| \frac{\partial \widehat{\phi}}{\partial \eta_j} (\zeta) \right| d\zeta \, ds,
\end{align*}
which is also finite, independent of $z$-variable and continuous in $\lambda$-variable. 

\medskip Similarly, one can see that $I_3 (\lambda, z)$ is equal to a constant multiple of $|\lambda|^{1/2} |z|$, thus completing the proof of claim in part $(ii)$. 

\medskip \noindent \textbf{Part $(iii)$:} As in the above part, working with $k=1$ and $P_1 (\lambda) = X_j (\lambda)$, it is clear that $I_2 (\lambda, z)$ and $I_3 (\lambda, z)$ tend to $0$ as $\lambda \to 0$. Thus, taking limit as $(\lambda, z) \to (0, z_0)$ in \eqref{eq:lim_lambda_and_z}, we get that 
\begin{align*}
\lim_{(\lambda, z) \to (0, z_0)} F(\lambda, z) &= \frac{-2\pi i}{\Gamma(1/2)} \int_{0}^{\infty} \int_{\C^n} t^{-1/2} \, e^{-2\pi i(x+2at, \, y+2bt) \cdot (\xi,\eta)} \, e^{-4\pi^2 t|\zeta|^2} \, \xi_j \, \widehat{\phi} \left( \zeta \right) d\zeta \, dt = F_0(z),  
\end{align*}
completing the proof of part $(iii)$. 
\end{proof}


\subsection{Proof of Theorem \ref{thm:Lp-Riesz-twisted-Laplacian}} \label{subsec:proof-two-basic-thms-1}

We shall first prove the case of $p=2$ in Theorem \ref{thm:Lp-Riesz-twisted-Laplacian}. Not surprisingly, the Euclidean Plancherel theorem is among our main tools. The proof is short and nice and moreover provides an important ingredient in the proof of the general case of $1 < p < \infty$. 
\begin{lemma} \label{lem:thm:Lp-Riesz-twisted-Laplacian-part-p=2}
For any $k \geq 1$, the Riesz transforms $\mathcal{R}_{k, \nu, \lambda}$ are bounded on $L^2(\C^n, d \mu_\nu)$, uniformly in $\nu$ and $\lambda$. 
\end{lemma} 
\begin{proof} 
Since $C_c^\infty(\C^n)$ is dense in $L^2(\C^n, d \mu_\nu)$, the lemma will follow from the Riesz representation theorem once we show that there exists a constant $C > 0$, independent of $\lambda$ and $\nu$, such that 
\begin{align} \label{ineq:p=2-boundedness-1}
\left| \left( \mathcal{R}_{k, \nu, \lambda} \phi, \, \psi \right) \right| = \left| \int_{\C^n} \left( \mathcal{R}_{k, \nu,\lambda} \phi \right)(z) \, \overline{\psi (z)} \, d\mu_\nu(z) \right| \leq C \|\phi\|_{L^2(\mathbb{C}^n, \, d\mu_\nu)} \, \|\psi\|_{L^2(\mathbb{C}^n, \, d\mu_\nu)}
\end{align}
for all $\phi, \psi \in C_c^\infty(\C^n)$. 

\medskip With $\phi, \psi \in C_c^\infty(\C^n)$ fixed, let us write $G(\lambda) = \left( \mathcal{R}_{k, \nu, \lambda} \phi, \, \psi \right)$. Then, it follows from Lemma \ref{lem:Riesz-lambda-pointwise-continuity} that $G$ extends to $\mathbb{R}$ as a continuous function. Thus, in view of the Riesz representation theorem, estimate \eqref{ineq:p=2-boundedness-1} is equivalent to 
\begin{align} \label{ineq:p=2-boundedness-2}
\left| \left( G, \, H \right) \right| \leq C \|\phi\|_{L^2(\mathbb{C}^n, \, d\mu_\nu)} \, \|\psi\|_{L^2(\mathbb{C}^n, \, d\mu_\nu)} \, \|H\|_{L^1(\mathbb{R})}
\end{align}
for every real-valued function $H \in L^1(\R)$. 

\medskip We now embark on a proof of \eqref{ineq:p=2-boundedness-2}. With $H$ a real-valued function in $L^1(\R)$, let us consider the following two functions: 
$$ H_1(\lambda)= |H(\lambda)|^{1/2} \quad \textup{and} \quad H_2(\lambda)=  \left\{
\begin{array}{ll}
+|H(\lambda)|^{1/2}, & \mbox{if } H(\lambda) \geq 0 \\
-|H(\lambda)|^{1/2}, & \mbox{if } H(\lambda) < 0.
\end{array} \right. $$ 
Clearly, $H_1, H_2 \in L^2(\R)$, $H(\lambda)= H_1(\lambda) H_2(\lambda)$, and $\|H_1\|_{L^2(\mathbb{R})} = \|H_2\|_{L^2(\mathbb{R})} = \|H\|_{L^1(\mathbb{R})}^{1/2}$.

\medskip Now, define $h_1, h_2 \in L^2(\mathbb{R})$ by $\widehat{h}_1 (\lambda) = H_1 (-2\pi \lambda)$ and $\widehat{h}_2 (\lambda) = H_2 (-2\pi \lambda)$, and consider functions $f$ and $g$ on $\mathbb{H}^n$ defined by 
$$ f(z,t) = \phi(z) h_1(t) \quad \textup{and} \quad g(z,t) = \psi(z)h_2(t), $$ 
so that $ f^\lambda(z)= \phi(z) H_1 (\lambda)$ and $g^\lambda(z)= \psi(z) H_2 (\lambda).$ Clearly,  
$$\|f\|_{L^2(\mathbb{H}^n, \, d\mu_\nu)}= \|\phi\|_{L^2(\mathbb{C}^n, \, d\mu_\nu)} \|h_1\|_{L^2(\mathbb{R})} \quad \textup{and} \quad \|g\|_{L^2(\mathbb{H}^n, \, d\mu_\nu)}= \|\psi\|_{L^2(\mathbb{C}^n, \, d\mu_\nu)} \|h_2\|_{L^2(\mathbb{R})}.$$ 

\medskip Now, as mentioned in \eqref{relation:Riesz-transform-gen-form-Cn-Hn}, we can associate with $\mathcal{R}_{k, \nu, \lambda}$ the Riesz transform $\widetilde{\mathcal{R}}_{k, \nu}$ by 
$$ \widetilde{\mathcal{R}}_{k, \nu} f (z,t) = \int_{\R} \left( \mathcal{R}_{k, \nu,\lambda} f^\lambda \right) (z) \, e^{-i \lambda t} \, d\lambda, $$
and using the Plancherel theorem for the Euclidean Fourier transform in $t$-variable, we get
\begin{align*}
\left( \widetilde{\mathcal{R}}_{k, \nu} f, \, g\right) = \int_{\mathbb{H}^n} \widetilde{\mathcal{R}}_{k, \nu} f (z,t) \, \overline{g(z,t)} \, dt \, d\mu_\nu(z) & = \frac{1}{2\pi} \int_{\C^n} \int_{\R} (\widetilde{\mathcal{R}}_{k, \nu}f)^\lambda (z) \,  \overline{g^\lambda(z)} \, d\lambda \, d\mu_\nu(z) \\ 
& = \int_{\C^n} \int_{\R} \left( \mathcal{R}_{k, \nu,\lambda} f^\lambda \right)(z) \, \overline{g^\lambda(z)} \, d\lambda \, d\mu_\nu(z) \\ 
& = \int_{\R} G(\lambda) \, H_1(\lambda) \, H_2(\lambda) \, d\lambda = \int_{\R} G(\lambda) \, H(\lambda) \, d\lambda. 
\end{align*}

\medskip Thus, making use of Theorem \ref{thm:LM-Lp-Riesz-Heisenberg}, we get 
\begin{align*}
|(G, \, H)| = \left| \left( \widetilde{\mathcal{R}}_{k, \nu} f, \, g \right) \right| & \leq \left\| \widetilde{\mathcal{R}}_{k, \nu}f \right\|_{L^2(\mathbb{H}^n, \, d\mu_\nu)} \, \|g\|_{L^2(\mathbb{H}^n, \, d\mu_\nu)} \\ 
& \lesssim \|f\|_{L^2(\mathbb{H}^n, \, d\mu_\nu)} \, \|g\|_{L^2(\mathbb{H}^n, \, d\mu_\nu)} \\ 
& = \|\phi\|_{L^2(\mathbb{C}^n, \, d\mu_\nu)} \, \|h_1\|_{L^2(\mathbb{R}^n)} \, \|\psi\|_{L^2(\mathbb{C}^n, \, d\mu_\nu)} \, \|h_2\|_{L^2(\mathbb{R})} \\ 
& = \|\phi\|_{L^2(\mathbb{C}^n, \, d\mu_\nu)} \, \| \widehat{h}_1 \|_{L^2(\mathbb{R}^n)} \, \|\psi\|_{L^2(\mathbb{C}^n, \, d\mu_\nu)} \, \| \widehat{h}_2 \|_{L^2(\mathbb{R})} \\ 
& = (2 \pi)^{-1} \|\phi\|_{L^2(\mathbb{C}^n, \, d\mu_\nu)} \, \|\psi\|_{L^2(\mathbb{C}^n, \, d\mu_\nu)}  \|H\|_{L^1(\mathbb{R})}, 
\end{align*}
which is the claimed estimate \eqref{ineq:p=2-boundedness-2}. 
\end{proof} 


We are now in a position to prove Theorem \ref{thm:Lp-Riesz-twisted-Laplacian} for any $1 < p < \infty$. 

\begin{proof}[Proof of Theorem \ref{thm:Lp-Riesz-twisted-Laplacian}]
We shall be closely following the terminology of the proof of Lemma \ref{lem:thm:Lp-Riesz-twisted-Laplacian-part-p=2}. With functions $\phi, \, \psi, \, G$ and the Riesz transform $\widetilde{\mathcal{R}}_{k, \nu}$ as in the proof of Lemma \ref{lem:thm:Lp-Riesz-twisted-Laplacian-part-p=2}, it suffices to prove that there exist a constant $C>0$, independent of $\lambda$ and $\nu$, such that 
\begin{align} \label{ineq:Riesz-twisted-Laplacian-Lp-bounded}
|G (\lambda)| = |\left( \mathcal{R}_{k, \nu, \lambda} \phi, \, \psi \right)| \leq C \|\phi\|_{L^p(\C^n, \, d\mu_\nu)} \, \|\psi\|_{L^{p^\prime}(\C^n, \, d\mu_\nu)}. 
\end{align}

\medskip For any $\epsilon > 0$ and $\lambda_0 \neq 0$, consider the following functions: 
$$ f_{\epsilon, \lambda_0} (z,t) := \epsilon^{\frac{1}{2p}} \, \phi(z) \, e^{-\epsilon \, t^2} \, e^{-i \lambda_0 t} \quad \textup{and} \quad g_{\epsilon, \lambda_0} (z,t) = \epsilon^{\frac{1}{2p'}} \, \psi(z) \, e^{-\epsilon \, t^2} \, e^{-i \lambda_0 t}. $$ 
Clearly, $f_{\epsilon, \lambda_0}^{\lambda}(z) = \pi^{1/2} \, \epsilon^{\frac{1}{2p}-\frac{1}{2}} \, \phi(z) \, e^{\frac{-(\lambda-\lambda_0)^2}{4 \epsilon}}$, $g_{\epsilon, \lambda_0}^\lambda(z) = \pi^{1/2} \, \epsilon^{\frac{1}{2p'}-\frac{1}{2}} \, \psi(z) \, e^{\frac{-(\lambda-\lambda_0)^2}{4 \epsilon}}$, and 
$$ \| f_{\epsilon, \lambda_0} \|_{L^p(\mathbb{H}^n, \, d\mu_\nu)} = (\pi/p)^{\frac{1}{2p}} \, \|\phi\|_{L^p(\C^n, \, d\mu_\nu)}, \quad \| g_{\epsilon, \lambda_0} \|_{L^{p^\prime}(\mathbb{H}^n, \, d\mu_\nu)} = (\pi/p')^{\frac{1}{2p'}} \, \|\psi\|_{L^{p^\prime}(\C^n, \, d\mu_\nu)}.$$ 

\medskip Again, by the Plancherel theorem for the Euclidean Fourier transform, we get 
\begin{align*}
\left( \widetilde{\mathcal{R}}_{k, \nu} f_{\epsilon, \lambda_0}, \, g_{\epsilon, \lambda_0} \right) & = \int_{\mathbb{H}^n} \widetilde{\mathcal{R}}_{k, \nu} f_{\epsilon, \lambda_0} (z,t) \, \overline{g_{\epsilon, \lambda_0} (z,t)} \, dt \, d\mu_\nu(z) \\ 
& = \frac{1}{2\pi} \int_{\C^n} \int_{\R} (\widetilde{\mathcal{R}}_{k, \nu} f_{\epsilon, \lambda_0})^\lambda (z) \,  \overline{g_{\epsilon, \lambda_0}^\lambda(z)} \, d\lambda \, d\mu_\nu(z) \\ 
& = \int_{\C^n} \int_{\R} \left( \mathcal{R}_{k, \nu,\lambda} f_{\epsilon, \lambda_0}^\lambda \right)(z) \, \overline{g_{\epsilon, \lambda_0}^\lambda(z)} \, d\lambda \, d\mu_\nu(z) \\ 
& = \pi (2\pi)^{1/2} \int_{\R} G(\lambda) \, (2\pi\epsilon)^{-1/2} \, e^{\frac{-(\lambda-\lambda_0)^2}{2 \epsilon}} \, d\lambda. 
\end{align*}

\medskip Now, thanks to Lemmas \ref{lem:Riesz-lambda-pointwise-continuity} and \ref{lem:thm:Lp-Riesz-twisted-Laplacian-part-p=2}, we know that $G$ extends as a continuous function to $\mathbb{R}$ which is also bounded on $\mathbb{R}$. Note also that the limit as $\epsilon \to 0$ in the final expression of the above integral exists as it is nothing but the convolution of the bounded continuous function $G$ with the approximate identity $k_\epsilon$ where $k_\epsilon(\lambda) = (2\pi\epsilon)^{-1/2} \, e^{\frac{-\lambda^2}{2\epsilon}}$. Thus, we get 
\begin{align} \label{ineq:Riesz-twisted-Laplacian-poinwise-limit} 
\left( \mathcal{R}_{k, \nu, \lambda_0} \phi, \, \psi \right) = G(\lambda_0) = \lim_{\epsilon \rightarrow 0} \pi^{-1} (2\pi)^{-1/2} \left( \widetilde{\mathcal{R}}_{k, \nu}f_{\epsilon, \lambda_0}, \, g_{\epsilon, \lambda_0} \right). 
\end{align}
But, we know from Theorem \ref{thm:LM-Lp-Riesz-Heisenberg} that 
$\displaystyle \left\| \widetilde{\mathcal{R}}_{k, \nu} f \right\|_{L^p(\mathbb{H}^n, \, d\mu_\nu)} \lesssim \left\| f \right\|_{L^p(\mathbb{H}^n, \, d\mu_\nu)}$, 
with the implicit constant independent of $\nu$. Using it, inequality \eqref{ineq:Riesz-twisted-Laplacian-poinwise-limit} implies that 
\begin{align*}
\left| \left( \mathcal{R}_{k, \nu, \lambda_0} \phi, \, \psi \right) \right| \lesssim \lim_{\epsilon \to 0} \left\| f_{\epsilon, \lambda_0} \right\|_{L^p(\mathbb{H}^n, \, d\mu_\nu)} \left\| g_{\epsilon, \lambda_0} \right\|_{L^{p^\prime}(\mathbb{H}^n, \, d\mu_\nu)} \lesssim
\left\| \phi \right\|_{L^p(\C^n, \, d\mu_\nu)} \left\| \psi \right\|_{L^{p^\prime}(\C^n, \, d\mu_\nu)}
\end{align*}
which is the claimed inequality \eqref{ineq:Riesz-twisted-Laplacian-Lp-bounded} and this completes the proof of Theorem \ref{thm:Lp-Riesz-twisted-Laplacian}.
\end{proof} 


\subsection{Proof of Theorem \ref{thm:higher-order-not-uniform}} \label{subsec:proof-two-basic-thms-2}

We shall show that Theorem \ref{thm:higher-order-not-uniform} can be deduced from Theorem \ref{thm:LS-Riesz-higher-order-sharp-Euclidean}. As in the previous subsection, Lemma \ref{lem:Riesz-lambda-pointwise-continuity} provides a technical ingredient in the process of the proof. 

\begin{proof}[Proof of Theorem \ref{thm:higher-order-not-uniform}] 
With $P_k (\lambda) = (\nu \cdot \nabla (\lambda))^k$, when $k \geq 3$, the claim of the theorem is that the weak-type $(1,1)$ boundedness of the Riesz transform $\mathcal{R}_{P_k, \nu, \lambda} = P_k (\lambda) (-L_{\nu,\lambda})^{-k/2}$ is not uniform in $\lambda$ and $\nu$. Thanks to the discussion in Subsection \ref{subsec:drift-change-with-rotation} and \ref{subsec:drift-change-with-scaling}, we can assume that $\nu = e_1$ and therefore it suffices to show that with $P_k (\lambda) = X_1(\lambda)^k$ the seminorm bound of the Riesz transform $\mathcal{R}_{P_k, e_1, \lambda} = P_k (\lambda) (-L_{e_1,\lambda})^{-k/2} = X_1(\lambda)^k (L_{e_1,\lambda})^{-k/2}$ is not uniform in $\lambda \to 0$. 

\medskip On contrary, assume that $\mathcal{R}_{P_k, e_1, \lambda}$ is weak-type $(1,1)$ bounded, with the seminorm being uniform in $\lambda \to 0$. That is, 
\begin{align} \label{ineq:uniform-bound-special-Riesz-transform}
\|\mathcal{R}_{P_k, e_1, \lambda}\|_{L^{1}(\mathbb{C}^n, \, d\mu_{e_1}) \to L^{1,\infty}(\mathbb{C}^n, \, d\mu_{e_1})} \leq C 
\end{align} 
with the constant $C$ being independent of $\lambda$. 

\medskip Let us denote by $\mathcal{R}_{D_k}$ the Riesz transform $D_k (-\Delta_{e_1})^{-k/2}$ corresponding to the operator $D_k = \partial_{x_1}^k $. Then, by part $(iii)$ of Lemma \ref{lem:Riesz-lambda-pointwise-continuity}, we have the following pointwise convergence: 
\begin{align} \label{eq:lambda-to-0-pointwise-limit}
\lim_{\lambda \to 0} \mathcal{R}_{P_k, e_1, \lambda} \phi (z) = \mathcal{R}_{D_k} \phi (z),  
\end{align} 
for every function $\phi \in C_c^\infty \left( \C^n \right)$. 

\medskip Making use of \eqref{eq:lambda-to-0-pointwise-limit} and Fatou's Lemma, we get 
\begin{align*} 
\left\|\mathcal{R}_{D_k} \phi \right\|_{L^{1,\infty}(\mathbb{C}^n, \, d\mu_{e_1})} = \left\| \lim_{\lambda \to 0} \left| \mathcal{R}_{P_k, e_1, \lambda} \phi \right| \right\|_{L^{1,\infty}(\mathbb{C}^n, \, d\mu_{e_1})} &\leq \liminf_{\lambda \to 0} \left\|\mathcal{R}_{P_k, e_1, \lambda} \phi \right\|_{L^{1,\infty}(\mathbb{C}^n, \, d\mu_{e_1})} \\ 
&\leq C \|\phi\|_{L^1(\mathbb{C}^n, \, d\mu_{e_1})}, 
\end{align*}
where the last inequality follows from the assumption \eqref{ineq:uniform-bound-special-Riesz-transform}. 

\medskip Thus, we get that 
\begin{align} \label{ineq:uniform-bound-special-Riesz-transform-2}
\left\|\mathcal{R}_{D_k} \phi \right\|_{L^{1,\infty}(\mathbb{C}^n, \, d\mu_{e_1})} \leq C \|\phi\|_{L^1(\mathbb{C}^n, \, d\mu_{e_1})}. 
\end{align}
But, with $k \geq 3$, Theorem \ref{thm:LS-Riesz-higher-order-sharp-Euclidean} asserts that inequality \eqref{ineq:uniform-bound-special-Riesz-transform-2} is not possible. Therefore, inequality \eqref{ineq:uniform-bound-special-Riesz-transform} is not possible, completing the claim of Theorem \ref{thm:higher-order-not-uniform}. 
\end{proof}


\section{Kernel estimates and Proofs of Theorems \ref{thm:higher-order-Riesz-uniform-1-1-log-space} and \ref{thm:higher-order-bdd-drift}} \label{sec:kernel-estimates}

In view of the discussion from Subsections \ref{subsec:drift-change-with-rotation} and \ref{subsec:drift-change-with-scaling}, from now on we shall only work with $\nu = e_1$. Over the next two subsections, we establish some local and global estimates of the integral kernel $\mathcal{R}_{k, e_1, \lambda}(z,w)$ which are then used in the proofs of Theorems \ref{thm:higher-order-Riesz-uniform-1-1-log-space} and \ref{thm:higher-order-bdd-drift} in the final Subsections \ref{subsec:proof-two-kernel-estimates-thms-1} and \ref{subsec:proof-two-kernel-estimates-thms-2}. We also record here the following asymptotics of the $\mu_{e_1}$-measure of balls (see (2.10) in \cite{Li-Sjogren-Wu-drift-Euclidean-Math-Z-2016}): 
\begin{align} \label{ineq:ball-volume-asymptote}
\mu_{e_1}(B(z,r)) \sim \left\{
\begin{array}{ll}
r^{2n} \, e^{2x_1} , & \mbox{if } r \leq 1 \\
r^{\frac{2n-1}{2}} \, e^{2(x_1+r)} ,& \mbox{if } r >1.
\end{array} \right. 
\end{align}

\medskip We also record here the following basic estimates which shall be used frequently:  
\begin{align} \label{estimates:hyperbolic-trignometric}
& |\sinh (s)| \geq |s|, \, \cosh (s) \geq 1, \, |s \, \coth (s)| \geq 1, \, \cosh (s) \sim e^{|s|}, \\ 
& \nonumber \lim_{s \to 0} s \, \coth (s) =  \lim_{|s| \to \infty} \coth (s) = 1, \, \text{and} \, \sup_{r>0} r^N e^{-r} < \infty \, \text{for any} \, N \in \mathbb{N}. 
\end{align}

\medskip It is convenient to work with the gradient vector fields in the complex form  
\begin{align} \label{def:standard-scaled-complex-vector-fields-twisted-laplacian}
Z_{j}(\lambda) &= \frac{1}{2} \left( X_j (\lambda) - i \, Y_j (\lambda) \right) = \frac{\partial}{\partial z_{j}} - \frac{\lambda}{4} \bar{z}_j, \\ 
\nonumber \textup{and} \qquad \bar{Z}_j(\lambda) &= \frac{1}{2} \left( X_j (\lambda) + i \, Y_j (\lambda) \right) = \frac{\partial}{\partial \bar{z_{j}}} + \frac{\lambda}{4} z_{j}. 
\end{align} 

\medskip Since $X_j (\lambda)$ and $Y_j (\lambda)$ can be expressed in terms of $Z_j(\lambda)$ and $\bar{Z}_j(\lambda)$, it suffices to work with a Riesz transform $\mathcal{R}_{P_k, e_1, \lambda}$ of the following form: 
\begin{align} \label{specific-form:Riesz-transform-sec:kernel-estimates}
\mathcal{R}_{P_k, e_1, \lambda} = P_k(\lambda) (-L_{e_1,\lambda})^{-k/2}, 
\end{align} 
where $P_k (\lambda)$ is a finite linear combination of monomials of the type $\theta_1(\lambda) \theta_2(\lambda) \cdots \theta_{k}(\lambda)$ with $\theta_j(\lambda)$ coming from the set $\{ Z_1(\lambda), \ldots, Z_n(\lambda), \bar{Z}_1(\lambda), \ldots, \bar{Z}_n(\lambda) \}$ for every $1 \leq j \leq k$. And, as earlier, for one such monomial at a time, we denote the associated Riesz transform by $\mathcal{R}_{k, e_1, \lambda}$.

\medskip It follows from the straightforward calculations that the kernel $\mathcal{R}_{k, e_1, \lambda}(z,w)$ is a finite linear combination of terms of the form 
\begin{align} \label{eq:general_term_Riesz_transform}
& \int_{0}^{\infty} t^{\frac{k}{2}-1} e^{-t-x_1-u_1} \left( \frac{\lambda}{\sinh (\lambda t)} \right)^n e^{-\frac{1}{4} \lambda \coth (\lambda t) \, |z-w|^2} \\ 
\nonumber & \qquad \left( \prod_{j=1}^n \lambda^{l_j} (\coth (\lambda t))^{d_j} \{\lambda (z_j-w_j)\}^{p_j} \{\lambda (\bar{z}_j-\bar{w}_j)\}^{q_j}\right) e^{\frac{i\lambda}{2} \Im(z \cdot \bar{w})} \, dt,   
\end{align}
where $l_j \leq k_j/2, \, d_j \leq l_j+p_j+q_j$, and $p_j+q_j \leq k_j-2l_j$. 


\subsection{Local estimates} \label{subsec:local-estimates}  
In this section, we estimate the integral kernel $\mathcal{R}_{k, e_1, \lambda}(z,w)$ and its derivatives in the local region, that is, when $|z-w| < 2$. 

\begin{lemma} \label{lem:ker-estimates}
The following estimates hold true uniformly in the region $|z-w| < 2$: 
\begin{align} 
|\mathcal{R}_{k, e_1, \lambda}(z,w)| & \lesssim \frac{1}{\mu_{e_1}(B(z,|z-w|))}, 
\label{ineq:ker-local-estimate} \\    
\sum_{j=1}^n \left\{ |Z_j \mathcal{R}_{k, e_1, \lambda}(z,w)| + |\bar{Z_j} \mathcal{R}_{k, e_1, \lambda}(z,w)| \right\} & \lesssim \frac{1}{|z-w| \, \mu_{e_1}(B(z,|z-w|))} 
\label{ineq:grad-local}, 
\end{align}
with the implicit constants independent of $\lambda.$
\end{lemma}
\begin{proof} 
\underline{Proof of \eqref{ineq:ker-local-estimate}}: Note that the term of the type \eqref{eq:general_term_Riesz_transform} is dominated by $E_{\lambda} (z,w)$, where 
\begin{align} 
\label{eq:general_term_Riesz_transform-2} 
E_{\lambda} (z,w) & = e^{-2 x_1} \int_{0}^{\infty} t^{\frac{k}{2}-1} \, e^{-t} \, e^{x_1-u_1} \left(\frac{\lambda}{\sinh (\lambda t)} \right)^{n} e^{-\frac{1}{4} \lambda \coth (\lambda t) |z-w|^2} \\ 
\nonumber & \qquad \qquad \{\lambda \coth (\lambda t)\}^l  \{\lambda \coth (\lambda t) |z-w| \}^m \, dt
\end{align} 
with $0 \leq l \leq k/2$ and $0 \leq m \leq k-2l.$
We decompose $E_\lambda (z,w)$ as the sum of $E_{\lambda, 1} (z,w)$ and $E_{\lambda, 2} (z,w)$ and estimate them separately as follows. Remember also that we are working under the assumption that $|z-w| < 2$. 
\begin{align*}
E_{\lambda,1} (z,w) & = e^{-2 x_1} \int_{0}^{|z-w|^2} t^{\frac{k}{2}-1} \, e^{-t} \, e^{x_1-u_1} \left(\frac{\lambda}{\sinh (\lambda t)} \right)^{n} e^{-\frac{1}{4} \lambda \coth (\lambda t) |z-w|^2} \\ 
& \qquad \qquad \{\lambda \coth (\lambda t)\}^l \{\lambda \coth (\lambda t) |z-w| \}^m \, dt \\ 
& \lesssim_k e^{-2 x_1} \int_{0}^{|z-w|^2} t^{\frac{k}{2}-1} \, e^{-t} \, t^{-n-l-m} \, (|z-w|^2 /t)^{-(n+1+m/2)} \, |z-w|^m \\ 
& \qquad \qquad \left(\frac{\lambda t}{\sinh (\lambda t)} \right)^{n} \{(\lambda t) \coth (\lambda t)\}^{l+m} \, dt.
\end{align*}
Now, we can make use of the fact that $\lambda t / \sinh (\lambda t) $ and $(\lambda t) \coth (\lambda t)$ are both bounded when $|\lambda t|$ is small, whereas when $|\lambda t|$ is large then $|\sinh (\lambda t)|$ grows like $e^{|\lambda t|}$ and $|\coth (\lambda t)| \lesssim 1$. In summary, $\displaystyle \sup_{\lambda, t} \left(\frac{\lambda t}{\sinh (\lambda t)} \right)^{n} \{(\lambda t) \coth (\lambda t)\}^{l+m} \leq C_{k} < \infty$. Therefore, we get 
\begin{align*}
E_{\lambda,1} (z,w) & \lesssim_k e^{-2 x_1} \int_{0}^{|z-w|^2} t^{\frac{k}{2}-1} \, e^{-t} \, t^{-n-l-m} \, (|z-w|^2 /t)^{-(n+1+m/2)} \, |z-w|^{m} \, dt \\ 
& = \frac{e^{-2 x_1}}{|z-w|^{2n+2}} \int_{0}^{|z-w|^2} t^{\frac{1}{2} (k-m-2l)} \, e^{-t} \, dt \\ 
& \lesssim \frac{e^{-2 x_1}}{|z-w|^{2n+2}} \int_{0}^{|z-w|^2} dt \lesssim \frac{e^{-2 x_1}}{|z-w|^{2n}} \sim \frac{1}{\mu_{e_1}(B(z,|z-w|))}. 
\end{align*}

\medskip Similarly, 
\begin{align*}
E_{\lambda, 2} (z,w) & = e^{-2 x_1} \int_{|z-w|^2}^{\infty} t^{\frac{k}{2}-1} \, e^{-t} \, e^{x_1-u_1} \left(\frac{\lambda}{\sinh (\lambda t)} \right)^{n} e^{-\frac{1}{4} \lambda \coth (\lambda t) |z-w|^2} \\ 
& \qquad \qquad \{\lambda \coth (\lambda t)\}^l  \{\lambda \coth (\lambda t) |z-w| \}^m \, dt \\ 
& \lesssim e^{-2 x_1} \int_{|z-w|^2}^{\infty} t^{\frac{k}{2}-1} \, e^{-t} \, t^{-n-l-m} \, (|z-w|^2 /t)^{-n-(m-1)/2} \, |z-w|^{m} \, dt \\ 
& = \frac{e^{-2 x_1}}{|z-w|^{2n-1}} \int_{|z-w|^2}^{\infty} t^{-3/2} \, t^{\frac{1}{2}(k-m-2l)} \, e^{-t} \, dt \\ 
& \lesssim \frac{e^{-2 x_1}}{|z-w|^{2n-1}} \int_{|z-w|^2}^{\infty} t^{-3/2} \, dt \lesssim \frac{e^{-2 x_1}}{|z-w|^{2n}} \sim \frac{1}{\mu_{e_1}(B(z,|z-w|))}. 
\end{align*} 
Above estimations of $E_{\lambda, 1} (z,w)$ and $E_{\lambda, 2} (z,w)$ together imply \eqref{ineq:ker-local-estimate}. 

\medskip \noindent \underline{Proof of \eqref{ineq:grad-local}}: We shall only work with the term $Z_j \mathcal{R}_{k, e_1, \lambda}(z,w)$ and other terms can be estimated analogously. It is easy to see from \eqref{eq:general_term_Riesz_transform} that a general term in $Z_j \mathcal{R}_{k, e_1, \lambda}(z,w)$ is a finite linear combination of terms of the same form as in \eqref{eq:general_term_Riesz_transform}, but this time with the condition on indices being $0 \leq l_j \leq (k_j+1)/2, \, d_j \leq l_j+p_j+q_j$, and $p_j+q_j \leq (k_j+1)-2l_j$. Each such term is dominated by $\widetilde{E}_{\lambda} (z,w)$, where 
\begin{align*}
\widetilde{E}_{\lambda} (z,w) & = e^{-2 x_1} \int_{0}^{\infty} t^{\frac{k}{2}-1} \, e^{-t} \, e^{x_1-u_1} \left(\frac{\lambda}{\sinh (\lambda t)} \right)^{n} e^{-\frac{1}{4} \lambda \coth (\lambda t) |z-w|^2} \\ 
& \qquad \qquad \{\lambda \coth (\lambda t)\}^l  \{\lambda \coth (\lambda t) |z-w| \}^m \, dt
\end{align*}
with $0 \leq l \leq (k+1)/2$ and $0 \leq m \leq k+1-2l.$ 

\medskip One can repeat the calculations from the proof of \eqref{ineq:ker-local-estimate} verbatim to show that 
\begin{align*}
\widetilde{E}_{\lambda} (z,w) \lesssim \frac{e^{-2 x_1}}{|z-w|^{2n+1}} \sim \frac{1}{|z-w| \, \mu_{e_1}(B(z,|z-w|))}, 
\end{align*}
thus establishing \eqref{ineq:grad-local}. 
\end{proof} 


\subsection{Estimates at infinity} \label{subsec:estimates-at-infinity} 
In this section, we estimate the integral kernel $\mathcal{R}_{k, e_1, \lambda}(z,w)$ in the region $|z-w| > 1$. 

\begin{lemma} \label{lem:ker-estimate-global}
The integral kernel $\mathcal{R}_{k, e_1, \lambda}(z,w)$ satisfies the following estimate uniformly in the region $|z-w| > 1$: 
\begin{align} \label{eq:global-estimate-for-bdd-drift-theorem}
|\mathcal{R}_{k, e_1, \lambda}(z,w)| \lesssim e^{-2x_1} \, e^{|z-w|} \, e^{-\frac{|\lambda|}{8} |z-w|^2}, 
\end{align} 
with the implicit constants being independent of $\lambda.$ 
\end{lemma}
\begin{proof}
It suffices to estimate term $E_{\lambda} (z,w)$ given by \eqref{eq:general_term_Riesz_transform-2}. This time we decompose $E_\lambda (z,w)$ as the sum of two terms $I_{\lambda, 1} (z,w)$ and $I_{\lambda, 2} (z,w)$ and estimate them separately as follows. In doing so, we shall again make use of the basic estimates \eqref{estimates:hyperbolic-trignometric}. Remember also that here we are working under the assumption that $|z-w| > 1$. 
\begin{align*}
I_{\lambda,1}(z, w) & = e^{-2 x_1} \int_{0}^{1/|\lambda|} t^{\frac{k}{2}-1} \, e^{-t} \, e^{x_1-u_1} \left(\frac{\lambda}{\sinh (\lambda t)} \right)^{n} e^{-\frac{1}{4} \lambda \coth (\lambda t) |z-w|^2} \\ 
& \qquad \qquad \{\lambda \coth (\lambda t)\}^l  \{\lambda \coth (\lambda t) |z-w| \}^m \, dt \\ 
& \lesssim e^{-2 x_1} \, e^{|z-w|} \int_{0}^{1/|\lambda|} t^{\frac{k}{2}-1} \, e^{-t} \, t^{-n-l-m} \, e^{-\frac{1}{8t}|z-w|^2} \, (|z-w|^2 / t)^{-n-l-m} \, |z-w|^{m} \, dt \\ 
& \lesssim \frac{e^{-2 x_1} \, e^{|z-w|}}{|z-w|^{2n+2l+m}} \, e^{-\frac{|\lambda|}{8} |z-w|^2} \int_{0}^{1/|\lambda|}  t^{k/2-1} \, e^{-t} \, dt \\ 
& \lesssim e^{-2 x_1} \, e^{|z-w|} \, e^{-\frac{|\lambda|}{8} |z-w|^2}. 
\end{align*}

\medskip Similarly, 
\begin{align*}
I_{\lambda,2}(z,w) & = e^{-2 x_1} \int_{1/|\lambda|}^{\infty} t^{\frac{k}{2}-1} \, e^{-t} \, e^{x_1-u_1} \left(\frac{\lambda}{\sinh (\lambda t)} \right)^{n} e^{-\frac{1}{4} \lambda \coth (\lambda t) |z-w|^2} \\ 
& \qquad \qquad \{\lambda \coth (\lambda t)\}^l  \{\lambda \coth (\lambda t) |z-w| \}^m \, dt \\ 
& \lesssim e^{-2 x_1} \, e^{|z-w|} \, |\lambda|^{n+l+m} \int_{1/|\lambda|}^\infty t^{k/2-1} \, e^{-t} \, e^{-\frac{|\lambda|}{8}|z-w|^2} \, (|\lambda| |z-w|^2)^{-n-l-m} \, |z-w|^{m} \, dt \\ 
& = e^{-2 x_1} \, e^{|z-w|} \, e^{-\frac{|\lambda|}{8} |z-w|^2} \, \frac{1}{|z-w|^{2n+2l+m}} \int_{1/|\lambda|}^\infty t^{k/2-1} \, e^{-t} \, dt \, \\ 
& \lesssim e^{-2 x_1} \, e^{|z-w|} \, e^{-\frac{|\lambda|}{8} |z-w|^2}. 
\end{align*}
This completes the proof of the claimed estimate \eqref{eq:global-estimate-for-bdd-drift-theorem}.
\end{proof}

Now, following \cite{Li-Sjogren-drift-sharp-endpoint-Euclidean-Canad-2021} (see identity (3.2) in \cite{Li-Sjogren-drift-sharp-endpoint-Euclidean-Canad-2021}), let us consider the following kernel: 
\begin{align} \label{def-kernel-Vk}
\mathcal{V}_k(z,w)= e^{-2x_1} \, |z-w|^{\frac{k-2n-1}{2}} \exp \left(-\frac{1}{4}\frac{|z^{\prime}-w^{\prime}|^2}{|z-w|}\right) \chi_{\{x_1-u_1>1\}},  
\end{align} 
where we are using the notation $z^\prime$ to denote the $(2n-1)$-coordinate vector obtained by removing the first coordinate $x_1$ from $z$, and define the operator
\begin{align} \label{def-operator-Vk}
\mathcal{V}_k f (z) = \int_{\C^n} \mathcal{V}_k(z,w) \, f(w) \, d\mu_{e_1}(w).  
\end{align} 

\medskip We shall make use of the following integral estimate (see Lemma 2.3 in \cite{Li-Sjogren-drift-sharp-endpoint-Euclidean-Canad-2021}). For any $s \in \mathbb{R}$ and $a>0$, let
$$ B_s(a) = \int_0^{+\infty} t^{s-1} \, e^{- t} \, e^{-\frac{a^2}{4 t}} \, dt. $$
We have 
\begin{align} \label{estimate:gamma-type-integral}
B_s(a) = \sqrt{2\pi} \, 2^{-s} \, a^{s-1/2} \, e^{-a} \, (1+O(a^{-1})), \quad \text{when} \, a \to \infty,
\end{align}
with the implicit constant depending on $s$ but not on $a$. 


\medskip We finish this subsection with the following basic estimate on the kernel $\mathcal{R}_{k, e_1, \lambda}(z,w)$. Let $k_j$ be the total power of $Z_j(\lambda)$ and $\Bar{Z}_j(\lambda)$ in the monomial $P_k(\lambda)$ defining $\mathcal{R}_{k, e_1, \lambda}$, so that $\sum_{j=1}^n k_j = k.$

\begin{lemma} \label{lem:V_{k_1}}
The integral kernel $\mathcal{R}_{k, e_1, \lambda}(z,w)$ satisfies the estimate
\begin{align} \label{est:lem:V_{k_1}} 
|\mathcal{R}_{k, e_1, \lambda}(z,w)| \lesssim \mathcal{V}_{k_1}(z,w), 
\end{align} 
uniformly in $\lambda \neq 0$ and $x_1-u_1>1$. 
\end{lemma}
\begin{proof}
Let $x_1-u_1>1.$ Recall the integral kernel of $(-L_{e_1,\lambda})^{-k/2}$ as given in \eqref{eq:kernel-fractional-power-twisted-lap-drift}. It can be easily 
verified via recursive calculations that the integral kernel $\mathcal{R}_{k, e_1, \lambda}(z,w)$ is a finite linear combination of the form 
$$\int_{0}^{\infty} t^{\frac{k}{2}-1} e^{-t-x_1-u_1} \left( \frac{\lambda}{\sinh (\lambda t)} \right)^n e^{-\frac{1}{4} \lambda \coth (\lambda t) \, |z-w|^2} \prod_{j=1}^n  \left(\prod_{l=1}^{m_j-p_j} A_{j,l} \right) \left(\prod_{l'=1}^{p_j} B_{j,l'} \right)
e^{\frac{i\lambda}{2} \Im(z \cdot \bar{w})} \, dt, $$
with $0 \leq m_j \leq k_j, \, m_j+p_j=k_j$ and $p_j \leq \text{min}\{m_j,\frac{k_j}{2}\}$. In the above expression, each $A_{j,l}$ is either $\left(-\frac{\delta_{1j}}{2}-\frac{\lambda\coth{(\lambda t)}}{4} (z_j-w_j)+\frac{\lambda}{4}(z_j-w_j)\right)$ or $\left(-\frac{\delta_{1j}}{2}-\frac{\lambda\coth{(\lambda t)}}{4} (\bar{z}_j-\bar{w}_j)-\frac{\lambda}{4} (\bar{z}_j-\bar{w}_j)\right)$ and $B_{j,l'}$ is either $\left(-\frac{\lambda}{4}\coth{(\lambda t)} + \frac{\lambda}{4}\right)$ or $\left(-\frac{\lambda}{4}\coth{(\lambda t)} - \frac{\lambda}{4}\right)$.

\medskip Note that each of the above integrals can be dominated by $I_{\lambda,1}(z,w)+I_{\lambda,2}(z,w)$, where
\begin{align*}
& I_{\lambda,1}(z,w) \\ 
&= \int_{0}^{\infty} t^{\frac{k}{2}-1} e^{-t-x_1-u_1} \left( \frac{\lambda}{\sinh (\lambda t)} \right)^n e^{-\frac{|z-w|^2}{4t} } |z_1-w_1|^{m_1-p_1} \prod_{j=2}^n|z_j-w_j|^{m_j-p_j} (\lambda \coth{(\lambda t)})^{\sum_{1}^n m_j} \, dt,
\end{align*} 
and 
$$ I_{\lambda,2}(z,w)= \int_{0}^{\infty} t^{\frac{k}{2}-1} e^{-t-x_1-u_1} \left( \frac{\lambda}{\sinh (\lambda t)} \right)^n e^{-\frac{|z-w|^2}{4t}} \prod_{j=2}^n|z_j-w_j|^{m_j-p_j} (\lambda \coth{(\lambda t)})^{\sum_{2}^n m_j+p_1} \, dt. $$

\medskip Let us first consider $I_{\lambda,1}(z,w)$.
\begin{align*}
& I_{\lambda,1}(z,w) \\ 
&= \int_{0}^{\infty} t^{\frac{k}{2}-1} e^{-t-x_1-u_1} \left( \frac{\lambda}{\sinh (\lambda t)} \right)^n e^{-\frac{|z-w|^2}{4t}} |z_1-w_1|^{m_1-p_1} \prod_{j=2}^n|z_j-w_j|^{m_j-p_j}(\lambda \coth{(\lambda t)})^{\sum_{1}^n m_j} \, dt \\ 
& \lesssim e^{-x_1-u_1} |z_1-w_1|^{m_1-p_1} |z^\prime-w^\prime| ^{\sum_{2}^n (m_j-p_j)} \int_{0}^{\infty} t^{\frac{k}{2}-n-\sum_{1}^n m_j} e^{-t} e^{-\frac{|z-w|^2}{4t}} \, \frac{dt}{t} \\ 
& \sim e^{-x_1-u_1-|z-w|} |z_1-w_1|^{m_1-p_1} |z^\prime-w^\prime| ^{\sum_{2}^n (m_j-p_j)} |z-w|^{\frac{k}{2}-n-\sum_{1}^n m_j-\frac{1}{2}} \\ 
& \lesssim e^{-x_1-u_1-|z-w|} \left(\frac{|z^\prime-w^\prime|^2}{|z-w|}\right)^{\frac{1}{2}\sum_{2}^n (m_j-p_j)} |z-w|^{\frac{k}{2}-n-p_1-\frac{1}{2}\sum_{2}^n (m_j+p_j)-\frac{1}{2}} \\ 
& \sim e^{-x_1-u_1-|z-w|}
\left(\frac{|z^\prime-w^\prime|^2}{|z-w|}\right)^{\frac{1}{2}\sum_{2}^n (m_j-p_j)} |z-w|^{\frac{k_1}{2}-n-p_1-\frac{1}{2}} \\ 
& \lesssim e^{-2x_1} \, |z-w|^{\frac{k_1-2n-1}{2}} \exp \left(-\frac{1}{4}\frac{|z^{\prime}-w^{\prime}|^2}{|z-w|}\right) = \mathcal{V}_{k_1}(z,w),
\end{align*}
where the third step follows from \eqref{estimate:gamma-type-integral} (with $s=\frac{k}{2}-n-\sum_{1}^n m_j$ and $a=|z-w|$), whereas in the second last step we have made use of the basic estimate 
$$ |z-w| - (x_1-u_1) = \frac{|z^\prime-w^\prime|^2}{|z-w|+(x_1-u_1)} \geq \frac{1}{2} \frac{|z^\prime-w^\prime|^2}{|z-w|}.$$

\medskip Similarly,
\begin{align*}
& I_{\lambda,2}(z,w) \\ 
&= \int_{0}^{\infty} t^{\frac{k}{2}-1} e^{-t-x_1-u_1} \left( \frac{\lambda}{\sinh (\lambda t)} \right)^n e^{-\frac{|z-w|^2}{4t}}  \prod_{j=2}^n|z_j-w_j|^{m_j-p_j}(\lambda \coth{(\lambda t)})^{\sum_{2}^n m_j+p_1} \, dt \\ 
& \lesssim e^{-x_1-u_1} |z^\prime-w^\prime| ^{\sum_{2}^n (m_j-p_j)} \int_{0}^{\infty} t^{\frac{k}{2}-n-\sum_{2}^n m_j-p_1} e^{-t} e^{-\frac{|z-w|^2}{4t}} \, \frac{dt}{t} \\ 
& \sim e^{-x_1-u_1-|z-w|}  |z^\prime-w^\prime| ^{\sum_{2}^n (m_j-p_j)} |z-w|^{\frac{k}{2}-n-\sum_{2}^n m_j-p_1-\frac{1}{2}} \\ 
& \lesssim e^{-x_1-u_1-|z-w|} \left(\frac{|z^\prime-w^\prime|^2}{|z-w|}\right)^{\frac{1}{2}\sum_{2}^n (m_j-p_j)} |z-w|^{\frac{k}{2}-n-p_1-\frac{1}{2}\sum_{2}^n (m_j+p_j)-\frac{1}{2}} \\ 
& \sim e^{-x_1-u_1-|z-w|}
\left(\frac{|z^\prime-w^\prime|^2}{|z-w|}\right)^{\frac{1}{2}\sum_{2}^n (m_j-p_j)} |z-w|^{\frac{k_1}{2}-n-p_1-\frac{1}{2}} \\ 
& \lesssim e^{-2x_1} \, |z-w|^{\frac{k_1-2n-1}{2}} \exp \left(-\frac{1}{4}\frac{|z^{\prime}-w^{\prime}|^2}{|z-w|}\right) = \mathcal{V}_{k_1}(z,w),
\end{align*}
where the third step follows from \eqref{estimate:gamma-type-integral} (with $s=\frac{k}{2}-n-\sum_{2}^n m_j-p_1$ and $a=|z-w|$).

\medskip This completes the proof of Lemma \ref{lem:V_{k_1}}. 
\end{proof}


\subsection{Proof of Theorem \ref{thm:higher-order-Riesz-uniform-1-1-log-space}} \label{subsec:proof-two-kernel-estimates-thms-1}

\begin{proof}[Proof of Theorem \ref{thm:higher-order-Riesz-uniform-1-1-log-space}] \label{proof-thm:higher-order-Riesz-uniform-1-1-log-space} Let $\mathcal{R}_{k,e_1,\lambda}$ be as in \eqref{specific-form:Riesz-transform-sec:kernel-estimates}. We split the Riesz transform into two parts: one with integral kernel supported locally, that is, only for $|z-w| < 2$ and the other one with the integral kernel supported at infinity, that is, only for $|z-w| \geq 1$. For the same, let $\eta \geq 0$ be a smooth function on $\C^n$ satisfying $\eta(z) = 1$ if $|z| \leq 1$ and $\eta(z) = 0$ if $|z| \geq 2$.

\medskip Now, we define the kernels $\mathcal{R}^{\infty}_{k, e_1, \lambda}(z,w) = \mathcal{R}_{k, e_1, \lambda}(z,w) \, (1-\eta(z-w))$ and $\mathcal{R}^{loc}_{k, e_1, \lambda}(z,w) = \mathcal{R}_{k, e_1, \lambda}(z,w) \, \eta(z-w)$, and the corresponding operators 
\begin{align*} 
\mathcal{R}^{loc}_{k, e_1, \lambda} \phi(z) & = \int_{\C^n} \phi(w) \, \mathcal{R}^{loc}_{k, e_1, \lambda}(z,w) \, d\mu_{e_1} (w), \quad \text{for} \, z \notin supp(\phi), \\ 
\text{and} \quad \mathcal{R}^{\infty}_{k, e_1, \lambda} \phi(z) & = \int_{\C^n} \phi(w) \, \mathcal{R}^{\infty}_{k, e_1, \lambda}(z,w) \, d\mu_{e_1} (w). 
\end{align*} 

\medskip Thanks to Lemma \ref{lem:ker-estimates}, we have that $\mathcal{R}^{loc}_{k, e_1, \lambda}(z,w)$ satisfies standard estimates of the Calder\'{o}n-Zygmund kernels with bounds independent of $\lambda$ and since $d\mu_{e_1}$ is a local doubling measure, we get that the operator $\mathcal{R}^{loc}_{k, e_1, \lambda}$ is of weak-type (1,1) with seminorm bounds being independent of $\lambda$. 

\medskip Finally, we are left with $\mathcal{R}^{\infty}_{k, e_1, \lambda}$. We write its kernel as the sum of two kernels, namely, 
\begin{align*}
\mathcal{R}^{\infty}_{k, e_1, \lambda}(z,w) &= \mathcal{R}^{\infty}_{k, e_1, \lambda}(z,w) \, \chi_{\{x_1-u_1 > 1\}} + \mathcal{R}^{\infty}_{k, e_1, \lambda}(z,w) \chi_{\{x_1-u_1 \leq 1\}} \\ 
& := \mathcal{R}^{\infty, 1}_{k, e_1, \lambda}(z,w) + \mathcal{R}^{\infty, 2}_{k, e_1, \lambda}(z,w). 
\end{align*} 

\medskip The operator $\mathcal{R}^{\infty, 2}_{k, e_1, \lambda}$ is much more better-behaved. More precisely, for any $k \geq 1$, it is of strong-type $(1,1)$ uniformly in $\lambda$. This will follow from the following bound 
$$ \sup_{w \in \mathbb{C}^n} \int_{\mathbb{C}^n} \left| \mathcal{R}^{\infty, 2}_{k, e_1, \lambda}(z,w) \right| d\mu_{e_1} (z) \leq C, $$
which we now prove. 

\medskip As earlier, working with a prototype term for the Riesz transform (with $0 \leq l \leq k/2$ and $0 \leq m \leq k-2l$), we have 
\begin{align*}
& \int_{\mathbb{C}^n} \left| \mathcal{R}^{\infty, 2}_{k, e_1, \lambda}(z,w) \right| d\mu_{e_1} (z) \\ 
& \lesssim \int_0^\infty \int_{\substack{|z-w|>1\\x_1-u_1\leq 1}} t^{\frac{k}{2}-1} \, e^{-t} \, t^{-n-l-m} \, e^{-\frac{1}{4t} |z-w|^2} \, |z-w|^{m} \, dz \, dt \\ 
& \lesssim \int_0^\infty \int_{|z-w|>1} t^{\frac{k}{2}-1} \, e^{-t} \, t^{-n-l-m} \, (|z-w|^2 / t)^{-\frac{1}{2} (2n+m+2)} \, |z-w|^{m} \, dz \, dt \\ 
& = \int_0^\infty t^{\frac{1}{2} (k-2l-m)} \, e^{-t} \, dt \, \int_{|z-w|>1} \frac{1}{|z-w|^{2n+2}} \, dz \leq C, 
\end{align*}
with $C$ independent of $\lambda$ as desired. 

\medskip So, we are only left to analyse the operator  $\mathcal{R}^{\infty, 1}_{k, e_1, \lambda}$ in various cases of $k_1 \geq 0$. 

\medskip \noindent \textbf{Proof of part $(1)$ of Theorem \ref{thm:higher-order-Riesz-uniform-1-1-log-space}:} We know from Proposition 3.1 of \cite{Li-Sjogren-drift-sharp-endpoint-Euclidean-Canad-2021} that the operators $\mathcal{V}_0$, $\mathcal{V}_1$ and $\mathcal{V}_2$ are of weak-type $(1,1)$. Therefore, it follows from Lemma \ref{lem:V_{k_1}} that for $k_1 \in \{0, 1, 2\}$, the operator $\mathcal{R}^{\infty, 1}_{k, e_1, \lambda}$ is of weak-type $(1,1)$, uniformly in $\lambda$. This completes the claim of part $(1)$ of Theorem \ref{thm:higher-order-Riesz-uniform-1-1-log-space}. 

\medskip \noindent \textbf{Proof of part $(2)$ of Theorem \ref{thm:higher-order-Riesz-uniform-1-1-log-space}:} For any $k_1 \geq 3$, it follows from Proposition 3.4 of \cite{Li-Sjogren-drift-sharp-endpoint-Euclidean-Canad-2021} that 
\begin{align*} 
\mu_{e_1} \{z: |\mathcal{V}_{k_1} f(z)|>\alpha\} \leq C_k \int_{\mathbb{C}^n} \frac{|f|}{\alpha} \left(1+ln^+\frac{|f|}{\alpha}\right)^{\frac{k_1}{2}-1} d\mu_{e_1}, 
\end{align*}
In view of the above estimate, Lemma \ref{lem:V_{k_1}} implies that 
\begin{align*} 
\mu_{e_1} \{z: |\mathcal{R}^{\infty, 2}_{k, e_1, \lambda} f(z)|>\alpha\} \leq C_k \int_{\mathbb{C}^n} \frac{|f|}{\alpha} \left(1+ln^+\frac{|f|}{\alpha}\right)^{\frac{k_1}{2}-1} d\mu_{e_1}, 
\end{align*}
uniformly in $\lambda$, completing the claimed estimate \eqref{ineq:thm-LlogL} in part $(2)$ of Theorem \ref{thm:higher-order-Riesz-uniform-1-1-log-space}. 

\medskip Finally, the sharpness of estimate \eqref{ineq:thm-LlogL} can be shown with the help of the transference technique. For the same, take the specific Riesz transform $T_\lambda = X_1(\lambda)^k (-L_{e_1,\lambda})^{-k/2}$. Note that in this example $k_1 = k\geq 3$. If there exists some $r < \frac{k_1}{2}-1$ such that 
\begin{align*} 
\mu_{e_1} \{z : | T_\lambda f(z)| > \alpha\} \leq C_{k,r} \int_{\mathbb{C}^n} \frac{|f|}{\alpha} \left(1+ln^+\frac{|f|}{\alpha}\right)^r d\mu_{e_1}, 
\end{align*}
for every $\alpha>0$, $\lambda \neq 0$ and $f \in L(1+ln^+L)^r (\mathbb{C}^n, \, d\mu_{e_1})$, then by Fatou's lemma 
\begin{align*} 
\mu_{e_1} \{z : | \partial_{x_1}^k (-\Delta_{e_1})^{-k/2} f(z)| > \alpha\} &= \mu_{e_1} \{z : \lim_{\lambda \to 0} | T_\lambda f(z)| > \alpha\} \\ 
& \leq \liminf_{\lambda \to 0} \mu_{e_1} \{z : | T_\lambda f(z)| > \alpha\} \\ 
& \leq C_{k,r} \int_{\mathbb{C}^n} \frac{|f|}{\alpha} \left(1+ln^+\frac{|f|}{\alpha}\right)^r d\mu_{e_1}, 
\end{align*}
for every $\alpha>0$, $f \in L(1+ln^+L)^r (\mathbb{C}^n, \, d\mu_{e_1})$. But that would contradict the assertion of Theorem \ref{thm:LS-Riesz-higher-order-sharp-Euclidean} and hence no $r < \frac{k_1}{2}-1$ will work in the estimate \eqref{ineq:thm-LlogL}. 
\end{proof}


\subsection{Proof of Theorem \ref{thm:higher-order-bdd-drift}} \label{subsec:proof-two-kernel-estimates-thms-2}

\begin{proof}[Proof of Theorem \ref{thm:higher-order-bdd-drift}] 
In view of the discussion from Sections \ref{subsec:drift-change-with-rotation} and \ref{subsec:drift-change-with-scaling}, the weak-type $(1,1)$ bound for $\mathcal{R}_{k, \nu, \lambda}$ is equivalent to that for $\mathcal{R}_{k, e_1, \lambda/|\nu|^2}$. Therefore, it suffices to show that for a given $M>0$ and $k \geq 3$, the operators $\mathcal{R}_{k, e_1, \delta}$ are of weak-type $(1,1)$ uniformly over $|\delta| \geq M$. 

\medskip We have already shown in the proof of Theorem \ref{thm:higher-order-Riesz-uniform-1-1-log-space} that the operator $\mathcal{R}^{loc}_{k, e_1, \delta}$ is of weak-type $(1,1)$ uniformly in $\delta$. Therefore, we will be done if we could show that $\mathcal{R}^{\infty}_{k, e_1, \delta}$ is of weak-type $(1,1)$ uniformly over $|\delta| \geq M$. In fact, $\mathcal{R}^{\infty}_{k, e_1, \delta}$ is of strong-type $(1,1)$ uniformly over $|\delta| \geq M$. 

\medskip From Lemma \ref{lem:ker-estimate-global}, we have for all $|\delta| \geq M$, 
\begin{align*}
|\mathcal{R}^{\infty}_{k, e_1, \delta}(z,w)| & \lesssim e^{-2x_1} \, e^{|z-w|} \, e^{-\frac{|\delta|}{8} |z-w|^2} \, \chi_{\{|z-w| > 1\}} \lesssim_M e^{-2x_1} \, e^{|z-w|} \, e^{\frac{-M}{8} |z-w|^2} \, \chi_{\{|z-w| > 1\}}, 
\end{align*}
and therefore 
\begin{align*}
\int_{\mathbb{C}^n} |\mathcal{R}^{\infty}_{k, e_1, \delta}(z,w)| \, d\mu_{e_1}(z) \lesssim_M \int_{|z-w|>1} e^{|z-w|} \, e^{\frac{-M}{8} |z-w|^2} \, dz = C_M < \infty, 
\end{align*}
which implies that $\mathcal{R}^{\infty}_{k, e_1, \delta}$ is of strong-type $(1,1)$ uniformly over $ |\delta| \geq M$. 
\end{proof}


\section*{Acknowledgements} 
We are grateful to Prof. Sundaram Thangavelu for a number of important discussions at various stages of the work. First author was supported by the Senior Research Fellowship from IISER Bhopal. 


\bibliographystyle{amsalpha}

\providecommand{\bysame}{\leavevmode\hbox to3em{\hrulefill}\thinspace}
\providecommand{\MR}{\relax\ifhmode\unskip\space\fi MR }
\providecommand{\MRhref}[2]{%
  \href{http://www.ams.org/mathscinet-getitem?mr=#1}{#2}
}
\providecommand{\href}[2]{#2}

\end{document}